 \documentclass[11pt]{amsart}
 
 \usepackage[foot]{amsaddr}

\usepackage[normalem]{ulem}

\usepackage{mathtools}
\usepackage{thmtools}
\declaretheoremstyle[headfont=\normalfont]{normalhead}
\usepackage{color}
\usepackage{graphicx}
\usepackage{morefloats}
\usepackage{hyperref}
\usepackage{ amssymb }
\usepackage{textcomp}
\usepackage{array}
\usepackage{float}
\usepackage{enumitem}
\usepackage[utf8]{inputenc}
\usepackage{amsthm}
\usepackage{chngcntr}

\usepackage{apptools}

\AtAppendix{\counterwithin{lemma}{section}}
\AtAppendix{\counterwithin{equation}{section}}
\AtAppendix{\counterwithin{theorem}{section}}
\AtAppendix{\counterwithin{proposition}{section}}
\AtAppendix{\counterwithin{definition}{section}}

\usepackage{pgfplots}
\usepackage{microtype}
\usepackage{comment}
\usepackage{tikz}
\usetikzlibrary{intersections,calc,arrows.meta,patterns,angles,quotes}

\providecommand{\dx}{\, \mathrm{d}x}
\providecommand{\ds}{\, \mathrm{d}s}

\providecommand{\omg}{\overline{\mathcal G}}
 
\providecommand{\dom}{\textup{dom}\,} 
\providecommand{\setR}{\mathbb{R}} 

\usepackage{tikz}
\usetikzlibrary{intersections,calc,arrows.meta,patterns,angles,quotes}

\usepackage{tikz}
\usepackage{pgfplots}
\usepackage{comment}
\usetikzlibrary{intersections,calc,arrows.meta,patterns,angles,quotes}
\usetikzlibrary{positioning}
\usetikzlibrary{backgrounds, pgfplots.fillbetween}
\usepackage{url}
\usepackage{float}
\usepackage[style=alphabetic,maxbibnames=5,backend=biber, doi=true,isbn=false,url=false,eprint=false,firstinits=true ]{biblatex}
\makeatletter
\long\def\unmarkedfootnote#1{{\long\def\@makefntext##1{##1}\footnotetext{#1}}}
\makeatother

\newtheoremstyle{mydef}
{\topsep}{\topsep}%
{}{}%
{\itshape}{}
{\newline}
{%
  \rule{\textwidth}{0.0pt}\\*%
  \thmname{#1}~\thmnumber{#2}\thmnote{\-\ #3}.\\*[-1.5ex]%
  \rule{\textwidth}{0.0pt}}%

\providecommand{\abstmp}[2]{{#1\lvert{#2}#1\rvert}}

\providecommand{\abs}[1]{\abstmp{}{#1}}

\providecommand{\vcr}{v_h}
\providecommand{\Inc}{\mathcal{I}_h}
\providecommand{\Anc}{W^{\textup{nc}}_h}
\providecommand{\nablanc}{\nabla_h}
\providecommand{\A}{W}
\providecommand{\ucr}{u_h}
\providecommand{\argmin}{\textup{argmin}}
\providecommand{\CR}{\textup{CR}^1(\mathcal{T}_h)}

\providecommand{\Lmdiv}{L_{\textup{div}}^{p_-'}(\Omega;\mathbb{R}^{m\times n})}
\providecommand{\Lpdiv}{L_{\textup{div}}^{p_+'}(\Omega;\mathbb{R}^{m\times n})}

\pgfplotsset{compat=1.16}

  \pgfplotscreateplotcyclelist{MyColors}{%
    {black,mark = *          ,every mark/.append style={solid,scale=0.8,fill=black}},
    {red,mark = diamond*			 ,every mark/.append style={solid,scale=0.8,fill=red}},
    {green,mark = square*			 ,every mark/.append style={solid,scale=0.8,fill=green}},
    {blue,mark = halfcircle*           ,every mark/.append style={solid,scale=0.8,fill=blue}},   
    {yellow,mark = otimes*           ,every mark/.append style={solid,scale=0.8,fill=black}},
    {dashed, black,mark = *          ,every mark/.append style={solid,scale=0.8,fill=black}},
    {dashed, red,mark = diamond*			 ,every mark/.append style={solid,scale=0.8,fill=red}},
    {dashed, green,mark = square*			 ,every mark/.append style={solid,scale=0.8,fill=green}},
    {dashed, blue,mark = halfcircle*           ,every mark/.append style={solid,scale=0.8,fill=blue}},    
    {yellow,mark = *           ,every mark/.append style={solid,scale=0.8,fill=yellow}},
    {red,mark = diamond*	 ,every mark/.append style={solid,scale=0.8,fill=red}},
    {red,mark = square*		 ,every mark/.append style={solid,scale=0.8,fill=red}},
    {red,mark = halfcircle*	 ,every mark/.append style={solid,scale=0.8,fill=red}},    
    {green,mark = diamond*	 ,every mark/.append style={solid,scale=0.8,fill=green}},
    {green,mark = square*		 ,every mark/.append style={solid,scale=0.8,fill=green}},
    {green,mark = halfcircle*	 ,every mark/.append style={solid,scale=0.8,fill=green}},   
    {blue,mark = diamond*    ,every mark/.append style={solid,scale=0.8,fill=blue}},
    {blue,mark = square* 	 ,every mark/.append style={solid,scale=0.8,fill=blue}},
    {blue,mark = halfcircle* ,every mark/.append style={solid,scale=0.8,fill=blue}},
    {dashed,violet},
    {dashed,teal},
    {dashed,purple},
    {dashed,cyan},
  }

\begin{document}

\newtheorem{theorem}{Theorem}
\newtheorem{conjecture}[theorem]{Conjecture}
\newtheorem{definition}[theorem]{Definition}
\newtheorem{proposition}[theorem]{Proposition}
\newtheorem{question}[theorem]{Question}
\newtheorem{remark}[theorem]{Remark}
\newtheorem{proposal}[theorem]{Proposal}
\newtheorem{lemma}[theorem]{Lemma}
\newtheorem{corollary}[theorem]{Corollary}
\newtheorem{observation}[theorem]{Observation}

\author[Kh.Balci]{Anna Kh.Balci}
\address{Fakult\"at für Mathematik, University Bielefeld,
Universit\"atsstrasse 25, 33615 Bielefeld, Germany}
\email{akhripun@math.uni-bielefeld.de}

\author[Ortner]{Christoph Ortner}
\address{University of British Columbia,
1984, Mathematics Road
Vancouver, BC, Canada }
\email{ortner@math.ubc.ca}

\author[Storn]{Johannes Storn}
\address{Fakult\"at für Mathematik, University Bielefeld,
Universit\"atsstrasse 25, 33615 Bielefeld, Germany}
\email{jstorn@math.uni-bielefeld.de}



\title[{Crouzeix-Raviart FEM for Lavrentiev gap}]{Crouzeix-Raviart finite element method for non-autonomous variational problems with Lavrentiev gap}

\begin{abstract}
    We investigate the convergence of the Crouzeix-Raviart finite element method for variational problems with non-autonomous integrands that exhibit non-standard growth conditions. While conforming schemes fail due to the Lavrentiev gap phenomenon, we prove that the solution of the Crouzeix-Raviart scheme converges to a global minimiser. Numerical experiments illustrate the performance of the scheme and give additional analytical insights.
\end{abstract}
\maketitle
\unmarkedfootnote {
\par\noindent  The research of A.\ Kh.Balci and J.\ Storn  was funded by German Research Foundation (DFG)  through the CRC 1283.
}

\section{Introduction}\label{sec:Introduction}
Many models in mechanics, optimal control, and other areas can be formulated in terms of variational problems seeking the minimiser of
\begin{align}\label{eq:Fun}
    \mathcal{F}(v)\coloneqq \int_\Omega \phi(x,\nabla v)\, - f \cdot v \dx\qquad\text{over all }v\in W \tag{$\mathcal{P}$}.
\end{align}
Our focus in the present work is on the numerical discretisation of \eqref{eq:Fun} in the case of non-autonomous integrands $\phi:\Omega \times \mathbb{R}^{m\times n} \to \mathbb{R}$, $m,n \in \mathbb{N}$, that are Carath\`eodory, convex in the second component, and satisfy the non-standard growth conditions with $1 < p_- < p_+ < \infty$ and 
\begin{align}
    \label{eq:growth_conditions}
    -c_0 +c_1 \abs{\xi}^{p_-}\le \phi(x,\xi)\le c_2\abs{\xi}^{p_+} +c_0.
\end{align}

We will assume throughout this paper that $\Omega \subset \mathbb{R}^n$ is a bounded polygonal domain (to allow for discretisation), $W \coloneqq  \lbrace v \in W^{1,p_-}(\Omega;\mathbb{R}^m) \colon v|_{\partial \Omega} = \psi|_{\partial \Omega}\rbrace$ with boundary data $\psi \in W^{1,p_+}(\Omega;\mathbb{R}^m)$, and right-hand side $f\in L^{p_-'}(\Omega;\mathbb{R}^m)$. 
Under these natural restrictions, Tonelli's theorem states well-posedness of \eqref{eq:Fun}. In the context of nonlinear  elasticity such problems were studied in~\cite{Ball82},~\cite{Ball01}.

A naive approach to discretise \eqref{eq:Fun} is the minimsation of the energy $\mathcal{F}$ over some discrete subspace $W_h \subset W$. However, non-standard growth conditions $p_-< p_+$ may lead to the Lavrentiev gap phenomenon \cite{Zhi83,Zhi86,Mar89}, that is, it may occur with $H \coloneqq W \cap W^{1,\infty}(\Omega;\mathbb{R}^m)$ that
\begin{align}\label{eq:LavrentievGap}
    \min_{v \in W} \mathcal{F}(v) < \inf_{w \in H} \mathcal{F}(w).
\end{align}
Since conforming discretisations generally satisfy $W_h \subset H$, the strict inequality in \eqref{eq:LavrentievGap} implies the failure of this scheme. 
Indeed, the approximated energy $\min_{W_h} \mathcal{F}$ converges to the so-called $H$-minimum $\inf_{H} \mathcal{F}$ as the underlying triangulation is refined \cite{CarstensenOrtner10}. 

Several numerical schemes have been suggested to overcome the Lavrentiev gap, see for example \cite{BallKnowles87,Li92,Li95,BaiLi06,CarstensenOrtner10,FengSchnake18}. These methods require a regularisation/penalisation leading to convergence only in a dual limit. The careful balancing of the regularisation parameter with refinement of the underlying triangulation is a challenging task that can be avoided, in some important situations, by the use of non-conforming finite element methods where $W_h \not\subset W$ \cite{Ort11,OrtnerPraetorius11}. This non-conforming approach leads to a numerical scheme that converges to the $W$-minimum $\min_W  \mathcal{F}$ for autonomous and convex integrands~$\phi(x,\xi)=\phi(\xi)$. 

In the present work we show how to adapt the numerical scheme as well as the convergence analysis for the much larger class of variational problems \eqref{eq:Fun} with \textit{non}-autonomous integrands. The main result in Theorem \ref{thm:main} states the convergence of the approximated minimiser and energy to a $W$-minimiser $u\in \argmin_W \mathcal{F}$ and the $W$-minimum $\min_{W} \mathcal{F}$, provided a specialised numerical quadrature rule is employed and under mild additional restrictions on the integrand (cf.\ \eqref{eq:Assumptionintegrand}--\eqref{eq:nabla2} or \eqref{eq:AlternaitveAss} on p.~\pageref{eq:Assumptionintegrand}).
These assumptions are valid for wide class of variational problems that exhibit a Lavrentiev gap, including important models such as the double-phase potential \cite{Zhi95,EspLeoMin04,FonMalMin04, ColMin15, BalDieSur20,BalSur21}, the variable exponent Laplacian \cite{Zhi86,Zhi95,Has05}, and the weighted $p$-energy \cite{HarHas19}. The problems cited above do  not exhaust the full range of present research in numerical analysis for problems with non-standard growth(for example, a posteriori estimates for variational problems with nonstandard power functionals~\cite{Pas20}).

\subsection*{Outline}
Section~\ref{sec:Discre} introduces the numerical scheme that utilises Crouzeix-Raviart functions and a specialised one-point quadrature rule. 
Section~\ref{sec:AssumptionOnTheIntegrand} provides two different sets (A) and (B) of assumptions. The set (A) contains, compared to the set (B), a weaker assumption on the quadrature and stronger assumptions on the integrand. 
The convergence proof of the numerical scheme follows in Section~\ref{sec:ProofOfConv}. The proof is quite direct under the assumption in (B). The relaxed assumption on the quadrature in (A) requires a more involved analysis, which is performed in Section~\ref{sec:ProofOfLemma3}. This proof exploits ideas of Zhikov \cite{Zhi11}, in particular it relies on the dual variational problem and the concept of relaxation~\eqref{eq:RelaxedG} allowing us to pass to the more regular problems.
We summarise examples for energies that experience a Lavrentiev gap in Section~\ref{sec:ExamplesLavrentiev}. 
Numerical experiments, displayed in Section~\ref{sec:exp}, apply our numerical scheme to these energies. In addition, the numerical experiments underline the importance of a suitable quadrature and  further investigate singularities and the Lavrentiev gap for energies with multiply saddle points.
We summarise required known results and calculations in the appendix. 

\section{Convergent Numerical Scheme}\label{sec:Crouzeix-Raviart FEM}
This section introduces the numerical scheme (Section~\ref{sec:Discre}) and proves its convergence (Section~\ref{sec:ProofOfConv}--\ref{sec:ProofOfLemma3}) under suitable assumptions on the integrand and quadrature (Section~\ref{sec:AssumptionOnTheIntegrand}).
\subsection{Discretisation}\label{sec:Discre}
Let $(\mathcal{T}_h)_{h>0}$ be a (not necessarily shape-regular and nested) sequence of regular triangulations of the domain $\Omega$ into simplices with diameters $\textup{diam}\, T < h$ for all $T\in \mathcal{T}_h$.
Let $\mathcal{E}_h$ denote the set of facets ($(d-1)$-subsimplices) of all cells $T\in \mathcal{T}_h$. The Crouzeix--Raviart space reads 
\begin{align*}
 \CR& \coloneqq \big\{ \vcr \in L^\infty(\Omega;\mathbb{R}^m) \colon \vcr|_T \text{ is affine for all }T\in \mathcal{T}_h\text{ and}\\
 &\qquad\qquad\qquad\text{continuous in the midpoints }\textup{mid}(e)\text{ for all }e\in \mathcal{E}_h\big\}. 
\end{align*}
Recall the boundary data $\psi \in W^{1,p_+}(\Omega;\mathbb{R}^m)$ from the definition of the space $W$ in \eqref{eq:Fun} and let $\mathcal{E}_h(\partial \Omega)$ denote the set of all facets $e\in \mathcal{E}_h$ on the boundary $\partial \Omega$. We define the space 
\begin{align*}
  \Anc \coloneqq \left\{ \vcr \in \CR \colon \vcr(\textup{mid}(e)) = \frac{1}{|e|} \int_e \psi \,\mathrm{d}s\text{ for all }e \in \mathcal{E}_h(\partial \Omega)\right\}.
\end{align*}
Notice that the Crouzeix-Raviart space is non-conforming in the sense that $\Anc \not\subset \A$. In particular, the gradient is only defined in the sense of distributions. However, it is possible to apply the gradient element-wise, that is, we set the broken gradient
\begin{align*}
    (\nabla_h v)|_T \coloneqq \nabla v|_T \qquad\text{for all }v \in \Anc + \A \text{ and }T\in \mathcal{T}. 
\end{align*}

Numerical schemes for problems with Lavrentiev gap appear to be very sensitive with respect to quadrature errors (cf.\ Section~\ref{subsec:BadQuad}).We will carefully analyse conditions under which the following simple one-point quadrature rule yields a convergent scheme: 
Given points $x_T \in T,$ $T \in \mathcal{T}_h$, 
we approximate $\phi$ by the mesh-dependent integrand
\begin{align*}
    \phi_h(x,\cdot) \coloneqq \phi(x_T,\cdot)\quad\text{for all } x \in T\in \mathcal{T}_h.
\end{align*}
This integrand is piece-wise constant in the first component and defines for all $v \in W + \Anc$ and $h>0$ the functional
\begin{align*}
    \mathcal{F}_h(v) \coloneqq \int_\Omega \phi_h(x,\nabla_h v) - f \cdot v\,\mathrm{d}x = \sum_{T\in \mathcal{T}_h} \int_T \phi(x_T,\nabla v) - f \cdot v\,\mathrm{d}x. 
\end{align*}
The resulting numerical scheme seeks a minimiser 
\begin{align}\label{eq:CRscheme}
\ucr \in \argmin_{\vcr \in \Anc} \, \mathcal{F}_h(\vcr). \tag{$\mathcal{P}_h^\textup{nc}$}
\end{align}
The existence of a discrete solution follows from the growth condition \eqref{eq:growth_conditions} and the direct method in calculus of variations.

\subsection{Main result}\label{sec:AssumptionOnTheIntegrand}
To state the convergence of the numerical scheme in \eqref{eq:CRscheme}, we introduce two alternative sets of assumptions on the integrand and the quadrature points $x_T \in T\in \mathcal{T}_h$.

\begin{itemize}
    \item (Quadrature) There exists a constant $c \in (0,1]$ such that we have for all $T \in \mathcal{T}_h$, $\xi \in \mathbb{R}^{m \times n}$, and $h>0$
\begin{align}\label{eq:Assumptionintegrand}
    c\, \phi(x_T,\xi) \leq \phi(x,\xi)  + 1\qquad\text{for almost all }x\in T. \tag{A1}
\end{align}
    Moreover, we assume the point-wise convergence $\phi_h(x,\xi) \to \phi(x,\xi)$ as $h\to 0$ for almost all $x\in \Omega$ and all $\xi \in \mathbb{R}^{m\times n}$.
\item ($\Delta_2$-condition) There exist constants $C,C_0 \in \mathbb{R}_{\geq 0}$ such that we have for almost all $x\in \Omega$ and all $\xi \in \mathbb{R}^{m\times n}$
        \begin{align}\label{eq:Delta_2phi}
        \phi(x, 2\xi)&\leq C\, \phi(x,\pm \xi) +C_0.\tag{A2}
    \end{align}
\item ($\nabla_2$-condition) If the constant $c$ in \eqref{eq:Assumptionintegrand} is smaller than $1$, there are constants $K,K_0 \in \mathbb{R}_{\geq 0}$ such that we have for almost all $x\in \Omega$ and all $\xi \in \mathbb{R}^{m\times n}$
\begin{align}\label{eq:nabla2}
    K\, \phi(x,2 \xi)& \leq \phi(x, \pm K \xi)+K_0.\tag{A3}
\end{align}
\end{itemize}
Notice that the $\nabla_2$-condition is equivalent to the $\Delta_2$-condition for the convex conjugate $\phi^*$, see Proposition \ref{lem:Nabla}.
The numerical experiment in Section \ref{subsec:BadQuad} shows that an assumption on the quadrature is necessary. The assumption in \eqref{eq:Assumptionintegrand} seems to be rather general, but might be relaxed if one uses additional information of specific energies. The relaxed assumption in \eqref{eq:Assumptionintegrand} comes with the price of a more involved convergence analysis. In particular, we investigate the dual problem, which requires the additional assumptions in \eqref{eq:Delta_2phi} and \eqref{eq:nabla2}.

We can circumvent this instructive but involved analysis by the following more restrictive assumption on the quadrature.
\begin{itemize}
\item There exists a constant $c_\phi<\infty$ such that, for all $h>0$ and $\xi \in \mathbb{R}^{m\times n}$ with $c_\phi<|\xi|$,
    \begin{align}\label{eq:AlternaitveAss}
        \phi_h(x,\xi) \leq \phi(x,\xi)\qquad\text{for almost all }x\in \Omega.    \tag{B}
    \end{align}
        Moreover, we suppose the point-wise convergence $\phi_h(x,\xi) \to \phi(x,\xi)$ as $h\to 0$ for almost all $x\in \Omega$ and all $\xi \in \mathbb{R}^{m\times n}$.
\end{itemize}
The following main result of this paper states convergence of the numerical scheme under the assumptions in \eqref{eq:Assumptionintegrand}--\eqref{eq:nabla2} or the assumption in \eqref{eq:AlternaitveAss}. Its proof is postponed to the following two subsections.
\begin{theorem}[Convergence]
\label{thm:main}
Suppose the integrand $\phi$ satisfies the two-sided growth conditions~\eqref{eq:growth_conditions} and  \eqref{eq:Assumptionintegrand}--\eqref{eq:nabla2} or \eqref{eq:AlternaitveAss}. Then the energies converge to the minimal energy, that is, we have 
\begin{align}
    \lim_{h \to 0}\mathcal{F}_h(u_h) = \mathcal{F}(u) = \min_{v\in \A} \mathcal{F}(v).\label{eq:Conv1}
\end{align}
Moreover, there exists a subsequence $(h_j)_{j\in\mathbb{N}}$ with $h_j \to 0$ such that 
\begin{align}
    u_{h_j}&\to u&&\text{strongly in }L^{p_-}(\Omega;\mathbb{R}^m),\label{eq:Conv2}\\
    \nabla_{h_j} u_{h_j}& \rightharpoonup \nabla u&&\text{weakly in }L^{p_-}(\Omega;\mathbb{R}^{m\times n}).\label{eq:Conv3}
\end{align}
If the minimiser $u\in \A$ is unique, the entire sequence converges. If the integrand is strictly convex, then the convergence is strong.
\end{theorem}
\subsection{Proof of Convergence}\label{sec:ProofOfConv}
The proof of Theorem \ref{thm:main} requires three preliminary results. The first one is an asymptotic lower bound for the computed energy. Its proof utilises the point-wise convergence in \eqref{eq:Assumptionintegrand} or \eqref{eq:AlternaitveAss}, that is,
\begin{align}\label{eq:PintWiseConv}
    \lim_{h\to 0} \phi_h(x,\xi) = \phi(x,\xi)\qquad\text{for all } \xi \in \mathbb{R}^{m\times n} \text{ and almost all }x \in \Omega.
\end{align}
A further tool is the conjugate functional theorem \cite[Chap.\ IX, (2.1)]{EkeTem76}, which involves the convex conjugate (with $\phi_0 \coloneqq \phi$)
\begin{align}\label{eq:defConvConj}
    \phi_h^*(\cdot ,\xi) \coloneqq \sup_{t \in \mathbb{R}^{m\times n}} \lbrace \xi:t - \phi_h(\cdot ,t) \rbrace\qquad\text{for all } \xi \in \mathbb{R}^{m\times n}\text{ and } h\geq 0.
\end{align}
The growth condition \eqref{eq:growth_conditions} and properties of the convex conjugate in Lemma \ref{lem:conj} yield for all $h\geq 0$, all $\xi \in \mathbb{R}^{m\times n}$, and almost all $x\in \Omega$ the growth
\begin{align}\label{eq:GrowthDual}
    -c_0+ c_2^{1-p_+'} |\xi|^{p_+'} \leq \phi_h^*(x,\xi) \leq c_1^{1-p_-'}|\xi|^{p_-'} + c_0.  
\end{align}
\begin{lemma}[Lower bound]
\label{lem:conv}
Let $(\vartheta_h)_{h>0} \subset L^1(\Omega;\mathbb{R}^{m\times n})$ be a weakly convergent sequence $\vartheta_h\rightharpoonup \vartheta$ in $L^1(\Omega;\mathbb{R}^{m\times n})$ as $h \to 0$. Then we have
\begin{align*}
    \int_\Omega \phi(x, \vartheta) \,\mathrm{d}x \leq \liminf_{h\to 0} \int_\Omega \phi_h(x, \vartheta_h)\,\mathrm{d}x.
\end{align*}
\end{lemma}
\begin{proof}
Let $(\vartheta_h)_{h>0} \subset L^1(\Omega;\mathbb{R}^{m\times n})$ be a weakly convergent sequence $\vartheta_h\rightharpoonup \vartheta$ in $L^1(\Omega;\mathbb{R}^{m\times n})$ as $h\to 0$. 
Young's inequality yields for all $z\in L^\infty(\Omega;\mathbb{R}^{m\times n})$ and $h\geq 0$ that
\begin{align}\label{eq:tempProof0}
    \int_\Omega z: \vartheta_h\, \mathrm{d}x-\int_\Omega \phi_h^*(x,z)\, \mathrm{d}x \leq \int_\Omega \phi_h(x,\vartheta_h)\, \mathrm{d}x.
\end{align}
Lemma~\ref{lem:conjuageconv} states that the point-wise convergence \eqref{eq:PintWiseConv} implies point-wise convergence of the conjugates, that is,
we have 
\begin{align*}
    \phi^*_h(x,\xi)\to \phi^*(x,\xi)\qquad\text{for all } \xi \in \mathbb{R}^{m\times n} \text{ and almost all }x \in \Omega.
\end{align*}
This point-wise convergence and the upper growth condition \eqref{eq:GrowthDual} for $\phi_h^*$ allow for the application of Lebesgue's theorem, which yields
\begin{align}\label{eq:TempProof1}
 \lim_{h\to 0}   \int_\Omega \phi_h^*(x,z) \dx = \int_\Omega \phi^*(x,z) \dx.
\end{align}
Taking the limit in \eqref{eq:tempProof0}, using the weak convergence in $L^1(\Omega;\mathbb{R}^{m\times n})$, and applying the identity in \eqref{eq:TempProof1} result in
\begin{align*}
\sup_{z\in L^\infty(\Omega;\mathbb{R}^{m\times n})} \int_\Omega z\colon \vartheta\dx-\int_\Omega \phi^*(x,z)\dx & \leq 
\liminf_{h\to 0} \int_\Omega \phi_h(x,\vartheta_h)\dx.
\end{align*}
The lemma follows from an application of the conjugate functional theorem \cite[Chap.\ IX, (2.1)]{EkeTem76}, which says
\begin{align*}
    &\sup_{z\in L^\infty(\Omega;\mathbb{R}^{m\times n})} \int_\Omega z\colon \vartheta\, \mathrm{d}x-\int_\Omega \phi^*(x,z)\, \mathrm{d}x = \int_\Omega \phi(x, \vartheta)\, \mathrm{d}x.\qedhere
\end{align*}
\end{proof}

In order to apply the previous lemma, we have to show that the gradients $(\nabla_h \ucr)_{h>0}$ of the discrete minimiser have a weakly convergent subsequence. This property follows from the following more general result.
\begin{lemma}[Convergent subsequence]\label{lem:ConvSubsequence}
    Let $\vcr \in \Anc$ be uniformly bounded in the sense that 
    \begin{align*}
        \sup_{h>0}\, \lVert \nablanc \vcr \rVert_{L^{p_-}(\Omega)} < \infty.
    \end{align*}
    Then there exists a sequence $(h_j)_{j\in \mathbb{N}}$ with $h_j \searrow 0$ and a function $v\in W$ with
    \begin{align*}
    \begin{aligned}
        v_{h_j} &\to v &&\text{strongly in }L^{p_-}(\Omega;\mathbb{R}^m), \\
        \nabla_{h_j} v_{h_j} &\rightharpoonup \nabla v&&\text{weakly in } L^{p_-}(\Omega;\mathbb{R}^{m \times n}).
    \end{aligned}
    \end{align*}
\end{lemma}
\begin{proof}
    Applying the result of \cite[Thm.\ 4.3]{Ort11} implies the statement but with $\nabla v_{h_j} \rightharpoonup \nabla v$ weakly in $L^1(\Omega;\mathbb{R}^{m \times n})$. The uniform $L^{p-}(\Omega;\mathbb{R}^{m \times n})$ bound on $\nabla v_{h_j}$ immediately implies that in fact 
    $\nabla v_{h_j} \rightharpoonup \nabla v$ weakly in $L^{p-}(\Omega;\mathbb{R}^{m \times n})$ as stated. The strong convergence of $v_{h_j}$ in $L^{p-}(\Omega;\mathbb{R}^{m})$ then follows by the compactness of the embedding of broken Sobolev spaces~\cite{2009-IMAJNA-DGFEM}.
\end{proof}

The final auxiliary result is an asymptotic upper bound for the minimal energy $\mathcal{F}_h(\ucr) = \min_{\Anc} \mathcal{F}_h$. 
\begin{lemma}[Upper bound]\label{lem:UpperBound}
Suppose \eqref{eq:Assumptionintegrand}--\,\eqref{eq:nabla2} or \eqref{eq:AlternaitveAss}, then we have
\begin{align*}
    \limsup_{h\to 0} \mathcal{F}_h(\ucr) \leq \min_{\A} \mathcal{F} = \mathcal{F}(u).
\end{align*}
\end{lemma}
\begin{proof}
\textit{Step 1 (Upper bound for $\mathcal{F}_h(\Inc v)$)}. 
Let $v\in W$ be arbitrary and set its non-conforming interpolation $\Inc v \in \Anc$ by
\begin{align*}
    \Inc v(\textup{mid}(e)) \coloneqq \frac{1}{|e|} \int_e v \,\mathrm{d}s\qquad\text{for all facets }e\in \mathcal{E}_h.
\end{align*}
An integration by parts reveals for all $T\in \mathcal{T}_h$ that 
\begin{align}\label{eq:IntMeanProp}
     \nabla (\Inc v)|_T = \frac{1}{|T|} \int_T \nabla v\dx. 
\end{align}
 Since $\phi_h|_T$ with $T\in \mathcal{T}_h$ is constant in the first component and convex in its second, Jensen's inequality and \eqref{eq:IntMeanProp} yield
\begin{align*}
    \int_T \phi_h(x,\nabla \Inc v)\,\mathrm{d}x \leq \int_T \phi_h(x, \nabla v)\,\mathrm{d}x.
\end{align*}
This estimate and the inequality $\lVert v - \Inc v \rVert_{L^{p_-}(T)} \leq C_\textup{apx}\, \textup{diam}(T) \, \lVert \nabla v\rVert_{L^{p_-}(T)}$ with $C_\textup{apx} = 1 +  2/n$ for all $T\in \mathcal{T}_h$ \cite[Lem.\ 2]{OrtnerPraetorius11} show
\begin{align}\label{eq:ProofTemp132}
    \mathcal{F}_h(\Inc v)& = \sum_{T\in \mathcal{T}_h} \int_T \phi_h(x,\nabla \Inc v) - f\cdot \Inc v\,\mathrm{d}x \\
    &\leq C_\textup{apx}\, h\, \lVert f \rVert_{L^{p_-'}(\Omega)} \,\lVert  \nabla v\rVert_{L^{p_-}(\Omega)} + \sum_{T\in \mathcal{T}_h} \int_T \phi_h(x,\nabla v) -  f\cdot v\,\mathrm{d}x.\notag
\end{align}
\textit{Step 2 (Proof with \eqref{eq:AlternaitveAss})}.
Suppose the assumption in \eqref{eq:AlternaitveAss} holds true with threshold $c_\phi<\infty$. 
Then the sum in the upper bound \eqref{eq:ProofTemp132} satisfies
\begin{align*}
    &\sum_{T\in \mathcal{T}_h} \int_T \phi(x_T,\nabla v) - f\cdot v \,\mathrm{d}x \\
    &= \sum_{T\in \mathcal{T}_h} \int_{T \cap \lbrace |\nabla v| \leq c_\phi\rbrace} \phi(x_T,\nabla v) \,\mathrm{d}x +\int_{T \cap \lbrace c_\phi < |\nabla v|\rbrace} \phi(x_T,\nabla v) \,\mathrm{d}x - \int_T f\cdot v \,\mathrm{d}x\\
     &  \leq \mathcal{F}(v) + \sum_{T\in \mathcal{T}_h} \int_{T \cap \lbrace |\nabla v| \leq c_\phi\rbrace} \phi(x_T,\nabla v) - \phi(x,\nabla v) \,\mathrm{d}x.
\end{align*}
The growth condition~\eqref{eq:growth_conditions} results for almost all $x\in T \cap \lbrace |\nabla u| \leq c_\phi\rbrace$ with $T\in \mathcal{T}_h$ in the upper bound
\begin{align*}
 |\phi(x_T,\nabla v(x))| \leq  c_2\abs{\nabla v(x)}^{p_+}+c_0 \leq  c_2c_\phi^{p_+}+c_0.
\end{align*}
Hence, Lebesgue's dominated convergence theorem shows 
\begin{align*}
    \sum_{T\in \mathcal{T}_h} \int_{T \cap \lbrace |\nabla u| \leq c_\phi\rbrace} \phi(x_T,\nabla v) - \phi(x,\nabla v) \,\mathrm{d}x \to 0 \qquad\text{as }h\to 0.
\end{align*}
This  concludes the proof under the assumption in \eqref{eq:AlternaitveAss}.\\[.5em]
\textit{Step 2' (Proof with \eqref{eq:Assumptionintegrand}--\eqref{eq:nabla2})}.
The inequality in \eqref{eq:ProofTemp132} reads 
\begin{align*}
    \mathcal{F}_h(\Inc v) &\leq \mathcal{F}_h(v) + C_\textup{apx}\, h\, \lVert f \rVert_{L^{p_-'}(\Omega)} \,\lVert  \nabla v\rVert_{L^{p_-}(\Omega)}.
\end{align*}
Taking the limit $h\to 0$ yields
\begin{align*}
        \limsup_{h\to 0} \min_{\Anc} \mathcal{F}_h \leq \limsup_{h\to 0} \min_{\A} \mathcal{F}_h.
\end{align*}
Thus, the lemma follows from the claim 
\begin{align}\label{eq:ClaimToShow}
    \limsup_{h\to 0} \inf_{\A} \mathcal{F}_h \leq \min_W \mathcal{F}.
\end{align}
The proof of the claim in \eqref{eq:ClaimToShow} is rather involved and, thus, postponed to the following Section \ref{sec:ProofOfLemma3}.
\end{proof}
After these three preliminary results we can prove this paper's main result. \begin{proof}[Proof of Theorem \ref{thm:main}]
Lemma \ref{lem:UpperBound} and the growth condition \eqref{eq:growth_conditions} lead for all sufficiently small $h>0$ to the upper bound
\begin{align*}
    &\sum_{T\in \mathcal{T}_h} \lVert \nabla \ucr\rVert_{L^{p_-}(T)}^{p_-} - c_0|T| - C_\textup{apx}\, h\,\lVert f \rVert_{L^{p_-'}(T)} \lVert \nabla \ucr \rVert_{L^{p_-}(T)}\\
    &\qquad\leq \sum_{T\in \mathcal{T}_h} \int_T \phi(x_T,\nabla \ucr)\,\mathrm{d}x - C_\textup{apx}\, h\,\lVert f \rVert_{L^{p_-'}(T)} \lVert \nabla \ucr \rVert_{L^{p_-}(T)}\\
    &\qquad \leq  \mathcal{F}_h(\ucr) \leq \mathcal{F}(u) + 1.
\end{align*}
In particular, we have for all sufficiently small $h>0$ a uniform upper bound
$
\lVert \nabla_h \ucr\rVert_{L^{p_-}(\Omega)}\leq C < \infty.
$
Hence Lemma \ref{lem:ConvSubsequence} yields the existence of a subsequence $(u_{h_j})_{j\in \mathbb{N}}$ and a function $v\in W$ with
\begin{align*}
    u_{h_j} \to v \text{ strongly in }L^{p_-}(\Omega;\mathbb{R}^m),\quad \nabla_{h_j} u_{h_j} \rightharpoonup \nabla v\text{ weakly in }L^{p_-}(\Omega;\mathbb{R}^{m\times n}).
\end{align*}
An application of Lemma \ref{lem:conv} and \ref{lem:UpperBound} shows that 
\begin{align*}
    \mathcal{F}(v) \leq \liminf_{j\to \infty} \mathcal{F}_{h_j}(u_{h_j}) \leq \limsup_{h\to 0} \mathcal{F}_h(u_h) \leq \mathcal{F}(u) = \min_{w\in \A} \mathcal{F}(w) < \infty.
\end{align*}
In particular, $v = u$ minimises the functional $\mathcal{F}$ over the set $\A$. This shows \eqref{eq:Conv2} and \eqref{eq:Conv3}. Since the arguments apply to any subsequence, the entire sequence $(\mathcal{F}_h(u_h))_{h>0}$ converges to $\mathcal{F}(u) = \min_{\A}\mathcal{F} = \lim_{h\to 0}\mathcal{F}_h(u_h)$. This shows \eqref{eq:Conv1}. 

If, in addition, the minimiser $u\in \A$  is unique, the same argument proves the convergence of the entire sequence 
$u_{h}\to u\text{ strongly in }L^{p_-}(\Omega;\mathbb{R}^m)$ and $\nablanc u_{h} \rightharpoonup \nabla u\text{ weakly in }L^{p_-}(\Omega;\mathbb{R}^{m\times n})$ as $h\to 0$. If the integrand is strictly convex in the second component then the strong convergence follows from \cite{Visintin}.
\end{proof}
\subsection{Proof of Lemma \ref{lem:UpperBound}}\label{sec:ProofOfLemma3}
We verify the claim in \eqref{eq:ClaimToShow} under the assumptions in \eqref{eq:Assumptionintegrand}--\eqref{eq:nabla2}, which we assume throughout this subsection. Our proof relies on techniques from \cite{Zhi11}. Among others, we exploit the dual formulation of the minimisation problem for all $h\geq 0$
\begin{align}\label{eq:MinPrimalh}
    \min_{v\in \A} \mathcal{F}_h(v). \tag{$\mathcal{P}_h$}
\end{align}
In order to include boundary data and right-hand side, we shift the functional $\mathcal{F}_h$ as follows. The surjectivity of the divergence operator~\cite{Bog80} yields the existence of a function $F \in L^{p_-'}(\Omega;\mathbb{R}^{m\times n})$ with $-\textup{div}\, F = f$.
We set $\phi_0 \coloneqq \phi$ and define for all $h\geq 0$ and $\xi \in \mathbb{R}^{m\times n}$ the (shifted) integrand 
\begin{align*}
\Phi_h(\cdot ,\xi) \coloneqq \phi_h(\cdot ,\xi) + F : \xi.
\end{align*}
Recall the boundary data $\psi \in W^{1,p_+}(\Omega;\mathbb{R}^m)$ and define the energy
\begin{align*}
    \hat{\mathcal{F}}_h(v) \coloneqq \int_\Omega \Phi_h(x,\nabla v + \nabla \psi) \dx\qquad\text{for all }v \in W^{1,1}(\Omega;\mathbb{R}^m).
\end{align*}
An integration by parts (with outer unit normal vector $\nu$) shows for all $v \in W_0^{1,p_-}(\Omega;\mathbb{R}^m)$ with homogeneous Dirichlet boundary data the identity 
\begin{align*}
    \hat{\mathcal{F}}_h( v) + \int_{\partial \Omega} \psi \cdot F\nu \ds = \mathcal{F}_h( v + \psi).
\end{align*}
In particular, we have for all $h\geq 0$ the equivalence of \eqref{eq:MinPrimalh} and the problem
\begin{align*}
    \min_{v\in W_0^{1,p_-}(\Omega;\mathbb{R}^m)} \hat{\mathcal{F}}_h( v)   +\int_{\partial \Omega} \psi \cdot F\nu \ds.
\end{align*}
The dual problem involves the convex conjugate~$\Phi_h^*$, which reads
\begin{align}\label{eq:MinPrimalMod}
\begin{aligned}
\Phi_h^*(\cdot,\zeta) &\coloneqq \sup_{\xi \in \mathbb{R}^{m\times n}} \lbrace \zeta:\xi - \Phi_h(\cdot,\xi)\rbrace= \sup_{\xi \in \mathbb{R}^{m\times n}} \lbrace (\zeta-F) : \xi - \phi_h(\cdot,\xi)\rbrace\\
&\hphantom{:}= \phi_h^*(\cdot,\zeta - F)\qquad\qquad\qquad\text{for all }\zeta \in \mathbb{R}^{m\times n}\text{ and }h\geq 0.
\end{aligned}
\end{align}
\begin{lemma}[Properties of $\Phi^*_h$]\label{lem:PropertiesOfG}
Let $h\geq 0$.
\begin{enumerate}
    \item There exist positive constants $\overline{c}_1,\overline{c}_2<\infty$ and a function $\overline{c}_0\in L^1(\Omega)$ such that the integrand $\Phi^*_h$ satisfies for almost all $x\in \Omega$ and all $\xi \in \mathbb{R}^{m\times n}$ the two sided growth condition\label{itm:1b}
\begin{align}\label{eq:duales}
    -\overline{c}_0(x) +\overline{c}_2\abs{\xi}^{p_+'}\le \Phi^*_h(x,\xi)\le \overline{c}_1\abs{\xi}^{p_-'} + \overline{c}_0(x).
\end{align}
\item With an $h$-independent constant $C< \infty$ it holds for almost all $x\in \Omega$ and all $\xi \in \mathbb{R}^{m\times n}$ that \label{itm:2b}
\begin{align}\label{eq:AssumptionDualFunctional}
    \Phi^*(x,\xi) - 1 \leq C\, \Phi^*_h(x,\xi).
\end{align}
\item  Let $(\tau_h)_{h>0}\subset L^1(\Omega;\mathbb{R}^{m\times n})$ be a weakly convergent sequence $\tau_h \rightharpoonup \tau$ in $L^1(\Omega;\mathbb{R}^{m\times n})$. Then we have\label{itm:3b}
\begin{align*}
   \int_{\Omega} \Phi^*(x,\tau)\dx \leq \liminf_{h\to 0} \int_\Omega \Phi^*_h(x,\tau_h)\dx.
\end{align*}
\end{enumerate}
\end{lemma}
\begin{proof}[Proof of \ref{itm:1b}]
Let $\xi \in \mathbb{R}^{m\times n}$ and $h\geq 0$. 
The growth condition in \eqref{eq:GrowthDual} shows  
\begin{align*}
    \phi_h^*(\cdot,\xi) \leq (-c_0 + c_1|\xi|^{p_-})^* = c_0 + c_1^{1-p_{-}'} \abs{\xi}^{p_-'}.
\end{align*}
This bound and the identity in \eqref{eq:MinPrimalMod} lead almost everywhere in $\Omega$ to 
\begin{align*}
    \Phi_h^*(\cdot,\xi) & = \phi_h^*(\cdot,\xi - F) \leq c_0 + c_1^{1-p_{-}'} \abs{\xi - F}^{p_-'} \\
    &\leq c_0 +2^{p_-'-1}c_1^{1-p_{-}'}( \abs{F}^{p_-'} + \abs{\xi}^{p_-'}).
\end{align*}
The lower bound in \eqref{eq:duales} follows similarly.\\[.5em]
\textit{Proof of \ref{itm:2b}}.
Let $\xi \in \mathbb{R}^{m\times n}$.
By the properties of the convex conjugate (see Lemma~\ref{lem:conj}) and the assumption in \eqref{eq:Assumptionintegrand} we have
\begin{align*}
    \phi^*(\cdot,\xi) \leq c\, \phi_h^*(\cdot,\xi/c) + 1\qquad\text{almost everywhere in }\Omega.
\end{align*}
If $c=1$, this yields $\phi^*(\cdot,\xi) \leq C \phi_h^*(\cdot,\xi) + 1$ with constant $C = 1$. If $c<1$, we use the $\nabla_2$-condition \eqref{eq:nabla2} (which yields the $\Delta_2$-condition for $\phi^*_h$, see Proposition \ref{lem:Nabla}) to conclude $\phi^*(\cdot,\xi) \leq C \phi_h^*(\cdot,\xi) + 1$ with some constant $C < \infty$. 
This and the identity in \eqref{eq:MinPrimalMod} yield \eqref{eq:AssumptionDualFunctional}.\\[.5em]
\textit{Proof of \ref{itm:3b}}.
The proof repeats the steps from the proof of Lemma~\ref{lem:conv}.
\end{proof}

We continue with the definition of the dual problem by setting the spaces $X \coloneqq L^{p_-}(\Omega;\mathbb{R}^{m\times n})$ and
$
    V \coloneqq \lbrace \nabla v \colon v\in W_0^{1,p_-}(\Omega;\mathbb{R}^m)\rbrace \subset X.
$
The orthogonal complement of $V$ reads
\begin{align*}
    V^\perp &\coloneqq \left\{ \tau \in X^* \colon \int_\Omega \tau  :\nabla v\dx = 0 \text{ for all } v\in W^{1,p_-}_0(\Omega;\mathbb{R}^m)  \right\}\\
    & \hphantom{:}= \left\{\tau \in L^{p_-'}(\Omega;\mathbb{R}^{m\times n})\colon \textup{div}\, \tau = 0 \right\} \eqqcolon \Lmdiv.
\end{align*}
By \cite[Chap.\ IX, (2.1)]{EkeTem76} we have for all $h\geq 0$ and $\tau \in L^{p_-'}(\Omega;\mathbb{R}^{m\times n})$ that
\begin{align*}
    \sup_{\xi \in L^{p_-}(\Omega;\mathbb{R}^{m \times n})} \int_\Omega \xi : \tau \dx - \int_\Omega \Phi_h(x,\xi) \dx  = \int_\Omega \Phi_h^*(x,\tau) \dx.
\end{align*}
This identity and an integration by parts lead to the dual functional
\begin{align}\label{eq:daasfsh}
\begin{aligned}
   &\mathcal{G}_h(\tau) = \hat{\mathcal{F}}_h^*(\tau) \coloneqq \sup_{\xi \in X} \int_\Omega \tau : \xi \dx - \int_\Omega \Phi_h(x,\xi + \nabla \psi) \dx \\ 
    & = \sup_{\vartheta \in X} \int_\Omega \tau : \vartheta \dx - \int_\Omega \Phi_h(x,\vartheta) \dx - \int_\Omega \tau : \nabla \psi\dx \\ 
    & = \int_\Omega \Phi_h^*(x,\tau) \dx- \int_\Omega \tau : \nabla \psi\dx\qquad\text{for all }\tau \in \Lmdiv.
\end{aligned}
\end{align}
\begin{lemma}[Equivalent problems]\label{lem:equiProbs}
It holds for all $h\geq 0$ that 
\begin{align*}
   \inf_{\Lmdiv} \mathcal{G}_h = - \min_{W_0^{1,p_-}(\Omega;\mathbb{R}^m)} \hat{\mathcal{F}}_h = \int_{\partial \Omega} \psi \cdot F\nu \ds - \min_W \mathcal{F}_h.
\end{align*}
\end{lemma}
\begin{proof}
The first identity follows from the classical convex optimisation theorem (see for example~\cite[Chap.\ III, Thm.\ 4.1]{EkeTem76}). The second identity results from the design of the (shifted) functional $\hat{\mathcal{F}}_h$.
\end{proof}
An advantage of the dual problem is that we can use the growth condition in \eqref{eq:duales} to apply Lebesgue's theorem, as done in the following lemma.
\begin{lemma}[Point-wise convergence]\label{lem:PointWiseConv}
We have 
\begin{align*}
    \lim_{h\to 0} \mathcal{G}_h(\tau) = \mathcal{G}(\tau)\qquad\text{for all }\tau \in \Lmdiv. 
\end{align*}
\end{lemma}
\begin{proof}
Let $\tau \in \Lmdiv$.
Due to the growth condition in \eqref{eq:duales} and the convergence result in Lemma \ref{lem:conjuageconv} (which can be applied due to the assumption in \eqref{eq:Assumptionintegrand}) we can apply Lebesgue's dominated convergence theorem to conclude $\int_\Omega \Phi^*(x,\tau) - \Phi^*_h(x,\tau)\dx \to 0$ as $h\to 0$. 
\end{proof}
A consequence of the point-wise convergence result is the following lemma.
\begin{lemma}[Upper bound for $\mathcal{G}_h$]\label{lem:UpperBoundForGh}
We have 
\begin{align*}
    \limsup_{h\to 0} \inf_{\Lmdiv} \mathcal{G}_h \leq \inf_{\Lmdiv} \mathcal{G}.
\end{align*}
\end{lemma}
\begin{proof}
Point-wise convergence (Lemma \ref{lem:PointWiseConv}) yields for all $\tau \in \Lmdiv$
\begin{align*}
    &\limsup_{h\to 0} \inf_{\Lmdiv} \mathcal{G}_h \leq \limsup_{h\to 0}\, \mathcal{G}_h(\tau) = \mathcal{G}(\tau).\qedhere
\end{align*}
\end{proof}
The following relaxation leads to an equivalent characterisation of the dual problem. We set for all $h\geq 0$ and $\tau \in \Lpdiv$ the relaxed energy
\begin{align}\label{eq:RelaxedG}
    \overline{\mathcal{G}}_h(\tau) \coloneqq \inf_{\substack{(\tau_k)_{k=1}^\infty \subset \Lmdiv\\ \lim_{k\to \infty} \lVert \tau_k - \tau \rVert_{L^{p_+'}(\Omega)} = 0}} \liminf_{k\to \infty} \mathcal{G}_h(\tau_k).
\end{align}
We denote by $\textup{dom}$ the effective domain of a functional, for example 
\begin{align*}
    \dom \overline{\mathcal{G}}_h = \lbrace \tau \in \Lpdiv \colon \overline{\mathcal{G}}_h(\tau) < \infty\rbrace\qquad\text{with }h\geq 0.
\end{align*}
\begin{lemma}[Properties of the relaxed energy functional $\overline{\mathcal{G}}_h$]\label{lem:PropOfG_h}
Let $h\geq 0$.
\begin{enumerate}
     \item It holds that \label{itm:1}
    \begin{align*}
        \mathcal{G}_h(\tau) & = \overline{\mathcal{G}}_h(\tau) &&\text{for all }\tau \in \Lmdiv,\\
        \mathcal{G}_h(\chi) & \leq \overline{\mathcal{G}}_h(\chi) &&\text{for all }\chi \in \Lpdiv.
    \end{align*}
    \item The minimum of $\overline{\mathcal{G}}_h$ over $\Lpdiv$ is attained and we have \label{itm:3}
    \begin{align*}
        \min_{\Lpdiv} \overline{\mathcal{G}}_h =  \inf_{\Lmdiv} \mathcal{G}_h.
    \end{align*}
       \item It holds that \label{itm:4}
    \begin{align*}
        \dom \overline{\mathcal{G}}_h \subset \dom \overline{\mathcal{G}}.
    \end{align*}
   \item For all $\tau \in \dom \overline{\mathcal{G}}_h$ we have \label{itm:5} 
   \begin{align*}
       \overline{\mathcal{G}}_h(\tau) = \mathcal{G}_h(\tau).
   \end{align*}
    \item The relaxed functional~$\omg_h$ is convex and weakly lower semi-continuous, that is, for all weakly convergent sequences $\tau_n \rightharpoonup \tau$ in~$\Lpdiv$   \label{itm:11}
\begin{align*}
     \omg_h(\tau) \leq \liminf_{n\to \infty} \omg_h(\tau_n).
\end{align*} 
\end{enumerate}
\end{lemma}
\begin{proof}[Proof of \ref{itm:1}]
Lemma~\ref{lem:PropertiesOfG}(\ref{itm:3b}) and the fact that strong convergence implies weak convergence show the inequality 
\begin{align*}
    \mathcal{G}_h(\chi) \leq \overline{\mathcal{G}}_h(\chi) \qquad\text{for all }\chi \in \Lpdiv.
\end{align*}
Equality follows for $\tau \in \Lmdiv$ by setting the constant sequence $\tau_k \coloneqq \tau$ for all $k\in \mathbb{N}$.\\[.5em]
%
%
\textit{Proof of \ref{itm:3}.}
Since $\omg_h$ satisfies the same growth conditions as $\mathcal{G}_h$ (cf.\ \eqref{eq:duales}),  Tonelli's theorem leads to the existence of a minimiser of $\overline{\mathcal{G}}_h$ in $\Lpdiv$.
Since we have the identity $\mathcal{G}_h = \overline{\mathcal{G}}_h$ for all $\tau \in \Lmdiv$, it holds that 
\begin{align*}
      \min_{\Lpdiv} \overline{\mathcal{G}}_h \leq \inf_{\Lmdiv} \mathcal{G}_h.
\end{align*}
On the other hand, for any sequence $(\tau_k)_{k=1}^\infty \subset \Lmdiv$ we have
\begin{align*}
    \inf_{\Lmdiv} \mathcal{G}_h \leq \liminf_{k\to \infty}  \mathcal{G}_h(\tau_k).
\end{align*}
Applying this observation to the definition of $\overline{\mathcal{G}}_h$ shows
\begin{align*}
      \min_{\Lpdiv} \overline{\mathcal{G}}_h \geq \inf_{\Lmdiv} \mathcal{G}_h.
\end{align*}
\textit{Proof of \ref{itm:4}.}
The inclusion follows from the estimate in~\eqref{eq:AssumptionDualFunctional}. \\[.5em]
\textit{Proof of \ref{itm:5}.} 
We set for all 
$\tau \in \Lpdiv$ the functional 
\begin{align*}
\mathcal{G}_h^0(\tau)&=\int_\Omega \Phi^*_h(x,\tau)\dx= \int_\Omega \phi^*_h(x,\tau -F)\dx.
\end{align*} 
 Let $x\in \Omega$ and $\xi \in \mathbb{R}^{m\times n}$. By definition we have 
 $
    \Phi_h(x,2\xi)=\phi_h(x,2\xi)+F(x):2\xi.   
 $
 Recall the constants $C,C_0$ in the $\Delta_2$-condition \eqref{eq:Delta_2phi}.
 Young's inequality shows for the second addend
         \begin{align*}
            F(x):2\xi &= (C+1)\, F(x):\xi + 2/(C+1)\,F(x): \xi\\
            &\leq (C+1)\, F(x):\xi + \phi_h(x,\xi) +  \phi_h^*(x,2/(C+1)\, F(x)).
        \end{align*}
        Set the function $C_1 \coloneqq C_0 + \phi_h^*(\cdot,2/(C+1)\, F) \in L^1(\Omega)$. Then
        the previous inequality and the $\Delta_2$-condition \eqref{eq:Delta_2phi} result in the $\Delta_2$-condition for $\Phi_h$
    \begin{align*}
            \Phi_h(x,2\xi) \leq C \phi_h(x,\xi) + C_0 + F(x):2\xi \leq (C+1)\, \Phi_h(x,\xi)\, + C_1(x).
    \end{align*}
Hence, Proposition~\ref{thm:Zhi1110} yields
    $
        \omg^0_h=\mathcal{G}^0_h$ on $\dom \omg^0_h. 
    $
    The definition of the energy $\mathcal{G}_h$ in \eqref{eq:daasfsh}, the definition of its relaxation $\omg_h$ in \eqref{eq:RelaxedG}, and the regularity of the boundary data $\psi \in W^{1,p_+}(\Omega;\mathbb{R}^{m})$ lead to $\dom \omg_h = \dom \omg^0_h$ and
    \begin{align*}
           \omg_h(\tau) = \omg_h^0(\tau) - \int_\Omega \nabla \psi:\tau\dx\qquad\text{for all }  \tau \in \Lpdiv. 
    \end{align*}
    Combining these observations concludes the proof.\\[.5em]
\textit{Proof of \ref{itm:11}}.
The convexity of $\omg_h$ follows from the convexity of $\mathcal{G}_h$.  
Moreover, \cite[Prop.~1.3.1]{But89} yields the lower semi-continuity of the relaxation functional $\omg_h$ with respect to strong convergence in $\Lpdiv$. The weak lower semi-continuity follows from the lower semi-continuity and convexity of $\omg_h$, see~\cite[Rem.\ 3.22]{Dac08}.
%
\end{proof}
The beneficial properties of $\overline{\mathcal{G}}_h$ allow us to prove the following result.
\begin{lemma}[Equality]\label{lem:Minimizer}
We have
\begin{align}\label{eq:tempIdetesa}
    \inf_{\Lmdiv}\mathcal{G} =\lim_{h\to 0} \inf_{\Lmdiv} \mathcal{G}_h.
\end{align}
 \end{lemma}
\begin{proof}
Let $(h_j)_{j=1}^\infty \subset \mathbb{R}_{>0}$ be a sequence with $h_j \searrow 0$ as $j\to \infty$. We denote for all $j\in \mathbb{N}$ by $\sigma_{j} \in \Lpdiv$ a minimiser of $\overline{\mathcal{G}}_{h_j}$.
By the growth conditions \eqref{eq:duales} the sequence $(\sigma_{j})_{j=1}^\infty  \subset \Lpdiv$ is uniformly bounded in $\Lpdiv$. Thus, there exists a function $\sigma \in \Lpdiv$ and a weakly convergent subsequence (which we do not relabel)
$\sigma_{j} \rightharpoonup \sigma$ in $\Lpdiv$ as $j\to \infty$.
By Lemma~\ref{lem:PropOfG_h}(\ref{itm:4}) we have~$\sigma_j \in \dom \omg_{h_j}\subset \dom \omg$. 
Lemma~\ref{lem:PropOfG_h}(\ref{itm:11}), the identity in \eqref{eq:AssumptionDualFunctional}, the fact that $\sigma_j$ is a minimiser, and the upper growth condition in \eqref{eq:duales} show 
\begin{align*}
    \omg(\sigma)\le \liminf_{j\to \infty}\omg(\sigma_j) \leq 
    \liminf_{j\to \infty}  C\,\omg_{h_j}(\sigma_j) + |\Omega| < \infty.
\end{align*}
Hence, we have $\sigma\in \dom \omg$ and so Lemma \ref{lem:PropOfG_h}(\ref{itm:5}) yields $\mathcal{G}(\sigma) = \overline{\mathcal{G}}(\sigma)$.  
This identity, $\sigma_j$ being a minimiser of $\mathcal{G}_{h_j}$, and the point-wise convergence (Lemma \ref{lem:PointWiseConv}) shows for all~$\tau \in \Lmdiv$
\begin{align}\label{eq:Proofasd}
    \overline{\mathcal G}(\sigma) = \mathcal{G}(\sigma) \leq \liminf_{j\to \infty} \mathcal G_{h_j}(\sigma_{j}) = \liminf_{j\to \infty} \overline{\mathcal{G}}_{h_j}(\sigma_{j}) \le  \liminf_{j\to \infty} \mathcal{G}_{h_j}(\tau) = \mathcal{G}(\tau).
\end{align}
This inequality and the identity in Lemma \ref{lem:PropOfG_h}\eqref{itm:3} result in  
\begin{align*}
   \min_{\Lpdiv}   \overline{\mathcal G} \leq   \overline{\mathcal G}(\sigma) \leq \inf_{\Lmdiv} \mathcal{G} = \min_{\Lpdiv}   \overline{\mathcal G}.
\end{align*}
Combining this inequality with \eqref{eq:Proofasd} shows 
\begin{align*}
    \inf_{\Lmdiv} \mathcal{G} = \overline{\mathcal{G}}(\sigma) \leq \liminf_{j\to 0} \inf_{\Lmdiv} \mathcal{G}_{h_j}.
\end{align*}
Since this result is true for subsequences of any sequence $h_j\searrow 0$, we have
\begin{align*}
    \inf_{\Lmdiv}\mathcal{G} \leq \liminf_{h\to 0} \inf_{\Lmdiv} \mathcal{G}_h.
\end{align*}
This and the upper bound in Lemma \ref{lem:UpperBoundForGh} lead to the identity in \eqref{eq:tempIdetesa}.
\end{proof}
After these preliminary results we can complete the proof Lemma \ref{lem:UpperBound}.
\begin{proof}[Proof of the claim in~\eqref{eq:ClaimToShow}]
The duality of the primal and dual problems (Lemma \ref{lem:equiProbs}), and the inequality in~\eqref{eq:tempIdetesa} in Lemma \ref{lem:Minimizer} lead to 
\begin{align*}
    &\limsup_{h \to 0} \mathcal{F}_h(u_h) \leq \limsup_{h \to 0} \min_{\A} \mathcal{F}_h = \limsup_{h \to 0} \int_{\partial \Omega} \psi \cdot F\nu \ds  - \inf_{\Lmdiv} \mathcal{\hat{F}}^*_h\\
    & = \int_{\partial \Omega} \psi \cdot F\nu \ds - \liminf_{h \to 0} \inf_{\Lmdiv} \mathcal{G}_h =\int_{\partial \Omega} \psi \cdot F\nu \ds  - \inf_{\Lmdiv} \mathcal{G}\\
    & = \min_{\A} \mathcal{F}.
\end{align*}
This concludes the proof of Lemma~\ref{lem:UpperBound} under \eqref{eq:Assumptionintegrand}--\eqref{eq:nabla2}. 
\end{proof}
\section{Examples on Lavrentiev Gap}\label{sec:ExamplesLavrentiev}
First examples of energies $\mathcal{F}$ with non-standard growth~\eqref{eq:growth_conditions} that experience a Lavrentiev gap go back to Zhikov~\cite{Zhi86}. A key idea in these examples is the construction of a function $u\in W$ which satisfies for any sequence $(u_n)_{n\in \mathbb{N}} \subset H$ with $u_n \to u$ in $W^{1,p_-}(\Omega;\mathbb{R}^m)$
\begin{align*}
\mathcal{F}(u_n) \to  \infty\qquad\text{as } n\to \infty.
\end{align*}
In the following examples we have~$m=1$, $n=2$, and $\Omega = (-1,1)^2$. For all $x=(x_1,x_2) \in \Omega$ the function $u$ reads~$u=(1-x_1^2-x_2^2)\, u_0(x)$ with\\
\begin{minipage}{0.65\textwidth}
\begin{align*}
  u_0(x) =\begin{cases} 1& \text{for } \abs{x_1} < x_2,\\
                        -1& \text{for }\abs{x_1} < -x_2,\\
                      x_2/\abs{x_1}& \text{else}.
             \end{cases}
  \end{align*}
\end{minipage}
\begin{minipage}{0.35\textwidth}
\centering
  \begin{tikzpicture}[scale=1.5]
    \node at (0,1.15) {$u_0$};
    \draw[dashed] (-1,-1) -- (-1,+1) -- (+1,+1) -- (+1,-1) --cycle;
    \draw (-1,-1) -- (1,1);
    \draw (1,-1) -- (-1,1);
    \node at (0,0.5) {$1$};
    \node at (0,-0.5) {$-1$};
\filldraw[pattern=vertical lines] (0,0)--(-1,1)--(-1,-1);
\filldraw[pattern=vertical lines] (0,0)--(1,1)--(1,-1);
\end{tikzpicture}
\end{minipage}
\ \\[-.5em]
Scaling the boundary data leads for sufficiently large scaling parameters $\lambda$
to the Lavrentiev gap phenomenon \cite[Sec.\ 3.3]{BalDieSur20}. Since the singularity of the function~$u$ is concentrated in the origin, examples of this type are called ``one saddle point'' or ``checker board'' setup. 

In the following we summarise known examples for non-autonomous problems caused by a single saddle point. In all these examples we have the 
boundary data $\psi = \lambda u_0$ with $\lambda>0$, the right-hand side $f \equiv 0$, and the following integrands for all $x=(x_1,x_2) \in \Omega$ and $\xi \in \mathbb{R}^2$.
\begin{enumerate}
    \item \textit{Piece-wise constant variable exponent (Zhikov~\cite{Zhi86})}. Let $1<p_-<2<p_+<\infty$, then the integrand reads $\phi(x,\xi)\coloneqq \abs{\xi}^{p(x)}/p(x)$ with\\
\begin{minipage}{0.60\textwidth}
\begin{align*}
     p(x)=\begin{cases}
       p_+\quad\text{ for }|x_1|< |x_2|,\\
       p_-\quad\text{ else}. 
      \end{cases}
 \end{align*}
 \end{minipage}
 \begin{minipage}{0.35\textwidth}
\begin{tikzpicture}[scale=1.5]
    \node at (0,1.15) {Exponent~$p$};
    \draw[dashed] (-1,-1) -- (-1,+1) -- (+1,+1) -- (+1,-1) --cycle;
    \draw (-1,-1) -- (1,1);
    \draw (1,-1) -- (-1,1);
    \node at (0,0.5) {$p_+$};
    \node at (0,-0.5) {$p_+$};
    \node at (-0.5,0) {$p_-$};
    \node at (0.5,0) {$p_-$};
  \end{tikzpicture}
\end{minipage}
\ \\[-.5em]
\item  \textit{Continuous variable exponent (Zhikov~\cite{Zhi95}, H\"asto~\cite{Has05})}.
As in the previous example we have the $p(\cdot)$-Laplacian $\phi(x,\xi) \coloneqq \abs{\xi}^{p(x)}/p(x)$ but with power (with linear interpolation in the dashed regions)\\
\begin{minipage}{0.60\textwidth}
 \begin{align*}
      p_\pm(x) = 2 \pm \frac{1}{2\,(\log(e+\abs{x}^{-1}))^{1/10}}.
  \end{align*}
  \end{minipage}
  \begin{minipage}{0.35\textwidth}
  \begin{tikzpicture}[scale=1.5]\node at (0,1.15) {Exponent~$p$};
      \draw[dashed] (-1,-1) -- (-1,+1) -- (+1,+1) -- (+1,-1) --cycle;
      \fill[pattern=north west lines] (1/2,1) -- (0,0) --
      (+1,1/2) -- (1,1) -- cycle;
      \draw(1/2,1) -- (0,0) -- (+1,1/2);
      \fill[pattern=north east lines] (1/2,-1) -- (0,0) --
      (+1,-1/2) -- (1,-1) -- cycle;
      \draw(1/2,-1) -- (0,0) -- (+1,-1/2);
      \fill[pattern=north east lines] (-1/2,1) -- (0,0) --
      (-1,1/2) -- (-1,1) -- cycle;
      \draw (-1/2,1) -- (0,0) -- (-1,1/2);
      \fill[pattern=north west lines] (-1/2,-1) -- (0,0) --
      (-1,-1/2) -- (-1,-1) -- cycle;
      \draw (-1/2,-1) -- (0,0) -- (-1,-1/2);
      
      \fill[white] (0,0) -- (1/2,1)--(-1/2,1)-- cycle;

      \fill[white] (0,0) -- (1/2,-1)--(-1/2,-1)-- cycle;

      \fill[white] (0,0) -- (1,1/2)--(1,-1/2)-- cycle;
      
      \fill[white] (0,0) -- (-1,1/2)--(-1,-1/2)-- cycle;
    \node at (0,.7) {$p_+$};
      \node at (0,-.7) {$p_+$};
      \node at (.6,0) {$p_-$};
      \node at (-.7,0) {$p_-$};
              \end{tikzpicture}
\end{minipage}
\ \\[-.5em]
\item \textit{Double phase potential (Zhikov~\cite{Zhi95}, Esposito--Leonetti--Mingione \cite{EspLeoMin04})}. This example involves the integrand $\phi(x,\xi)=1/p_-\, \abs{\xi}^{p_-}+a(x)/p_+\,\abs{\xi}^{p_+}$ with powers $1<p_-<2<2+\alpha<p_+<\infty$, where the constant $\alpha\geq 0$ enters the weight \\
\begin{minipage}{0.60\textwidth}
 \begin{align*}
     a(x)=\begin{cases}
       \abs{x_2}^\alpha&\text{for }|x_1| < |x_2|,\\
       0& \text{else}.
     \end{cases}
 \end{align*}   
 \end{minipage}
\begin{minipage}{0.35\textwidth}
 \begin{tikzpicture}[scale=1.5]\node at (0,1.15) {Weight~$a$};
      \draw[dashed] (-1,-1) -- (-1,+1) -- (+1,+1) -- (+1,-1) --cycle;
    \draw (-1,-1) -- (1,1);
    \draw (1,-1) -- (-1,1);
    \node at (0,0.5) {$\abs{x_2}^\alpha$};
    \node at (0,-0.5) {$\abs{x_2}^\alpha$};
    \node at (-0.5,0) {$0$};
    \node at (0.5,0) {$0$};
\end{tikzpicture}
\end{minipage}
\ \\[-.5em]
\item \textit{Borderline case of double phase potential (Balci--Surnachev~\cite{BalSur21})}.
This example involves constants $\beta>1$, $\gamma >1$ and the weight function $a$ from the previous example with $\alpha = 0$. The integrand reads
\begin{align*}
     \phi(x,\xi)=\log^{-\beta}(e+\abs{\xi}) \abs{\xi}^2+a(x)\log^{\gamma}(e+\abs{\xi}) \abs{\xi}^2.
 \end{align*}
\end{enumerate}
Further examples involve fractal contact sets \cite{BalDieSur20} or matrix-valued integrands \cite{FosHruMiz03}. The latter example was already treated numerically in \cite{Ort11}.

\section{Experiments}\label{sec:exp}
In this section we investigate the Lavrentiev gap phenomenon numerically. 
We approximate the $W$-solution by the non-conforming scheme introduced in Section \ref{sec:Crouzeix-Raviart FEM} and the $H$-solution by a conforming scheme with exact (or at least more accurate) quadrature. The convergence of the latter scheme with exact quadrature to the $H$-minimiser has been shown in \cite{CarstensenOrtner10}.
We denote the solutions to the non-conforming scheme by $u_\textup{nc}\in \CR$ and to the conforming scheme by $u_\textup{c}\in \mathcal{L}^1_1(\mathcal{T})$, where $\mathcal{L}^1_1(\mathcal{T})$ denotes the Lagrange space of continuous and piece-wise affine functions. The corresponding energies read $\mathcal{F}_h(u_\textup{nc})$ and $\mathcal{F}(u_\textup{c})$.
Our computations utilise an adaptive mesh refinement strategy driven by the residual-based error indicator introduced in \cite{OrtnerPraetorius11} for the non-conforming scheme. If not mentioned specifically, we solve the non-linear systems with a Newton scheme. 
\subsection{Model Problems}
In this subsection we apply our numerical scheme to the minimisation problems introduced in Section \ref{sec:ExamplesLavrentiev}.
These problems experience a Lavrentiev gap for all sufficiently large  parameters $\lambda > \lambda_0$ in the boundary data $\psi = \lambda u_0$ with some unknown threshold $\lambda_0\geq 0$. We use our numerical method to explore this threshold. This visualises the performance of our scheme and provides some insights for further analytical investigation.

\subsubsection*{Experiment 1 (Piece-wise constant variable exponent)}
In our computations for the first example in Section \ref{sec:ExamplesLavrentiev} we set the exponents $p_- = 3/2$ and $p_+=3$. Our initial triangulation resolves the domains for $p_-$ and $p_+$, which allows for an exact quadrature with piece-wise constant exponents. In particular, \eqref{eq:Assumptionintegrand}--\eqref{eq:nabla2} as well as \eqref{eq:AlternaitveAss} are obviously satisfied.
For the boundary data $\psi = \lambda u_0$ with $\lambda =1$ the exact solution reads $u(x_1,x_2) = x_2$ for all $(x_1,x_2) \in \Omega$. The convergence history plots, displayed on the left-hand side of Figure \ref{fig:exp1}, indicate that there is no Lavrentiev gap for $\lambda\leq 1$ and that there is a gap for $\lambda >1$.
Moreover, the convergence history plot suggests the speed of convergence $\min_W \mathcal{F} - \mathcal{F}(u_\textup{nc}) = \mathcal{O}(\textup{ndof}^{-1})$ and $\mathcal{F}(u_\textup{c}) - \min_H \mathcal{F} = \mathcal{O}(\textup{ndof}^{-1})$.
\begin{figure}
\begin{center}
\begin{tikzpicture}
\begin{axis}[
xmode = log,
ymode = log,
xlabel={ndof},
cycle multi list={\nextlist MyColors},
scale = {0.68},
legend cell align=left,
legend style={font=\scriptsize, legend columns=1,legend pos=south west}
]
	\addplot table [x=ndof,y=DistVal] {ExperimentsPaper/Data/Exp1/t_0.99_bulk_0.3Calc.txt};
		\addplot table [x=ndof,y=DistVal] {ExperimentsPaper/Data/Exp1/t_1.02_bulk_0.3Calc.txt};
		\addplot table [x=ndof,y=DistVal] {ExperimentsPaper/Data/Exp1/t_1.04_bulk_0.3Calc.txt};
	\addplot table [x=ndof,y=DistVal] {ExperimentsPaper/Data/Exp1/t_1.06_bulk_0.3Calc.txt};
	\addplot table [x=ndof,y=DistVal] {ExperimentsPaper/Data/Exp1/t_1.08_bulk_0.3Calc.txt};
		\addplot[dashed,sharp plot,update limits=false] coordinates {(1e2,8e-5) (1e7,8e-10)};
		\addplot[dashed,sharp plot,update limits=false] coordinates {(1e2,3e-6) (1e7,3e-11)};
		\legend{
	{$\lambda = 0.99$},{$\lambda = 1.02$},{$\lambda = 1.04$},{$\lambda = 1.06$},{$\lambda = 1.08$}};
\end{axis}
\end{tikzpicture}
\begin{tikzpicture}
\begin{axis}[
xmode = log,
ymode = log,
xlabel={ndof},
cycle multi list={\nextlist MyColors},
scale = {0.68},
legend cell align=left,
legend style={font=\scriptsize, legend columns=1,legend pos=north east}
]
	\addplot table [x=ndof,y=DistVal] {ExperimentsPaper/Data/Exp3/t_0.3_bulk_0.3Calc.txt};
		\addplot table [x=ndof,y=DistVal] {ExperimentsPaper/Data/Exp3/t_0.4_bulk_0.3Calc.txt};
	\addplot table [x=ndof,y=DistVal] {ExperimentsPaper/Data/Exp3/t_0.5_bulk_0.3Calc.txt};
	\addplot table [x=ndof,y=DistVal] {ExperimentsPaper/Data/Exp3/t_0.6_bulk_0.3Calc.txt};
		\addplot[dashed,sharp plot,update limits=false] coordinates {(1e2,1e-4) (1e7,1e-9)};\label{line:dashed}
		\addplot[dashed,sharp plot,update limits=false] coordinates {(1e2,4e-4) (1e7,4e-9)};
		\legend{
	{$\lambda = 0.3$},{$\lambda = 0.4$},{$\lambda = 0.5$},{$\lambda = 0.6$}};
\end{axis}
\end{tikzpicture}
\caption{Convergence history plot of the distances $\mathcal{F}(u_\textup{c})-\mathcal{F}(u_\textup{nc})$ in Experiment 1 (left) and Experiment 3 (right), where the dashed line \ref{line:dashed} indicates the rate $\mathcal{O}(\textup{ndof}^{-1})$}\label{fig:exp1}
\end{center}
\end{figure}
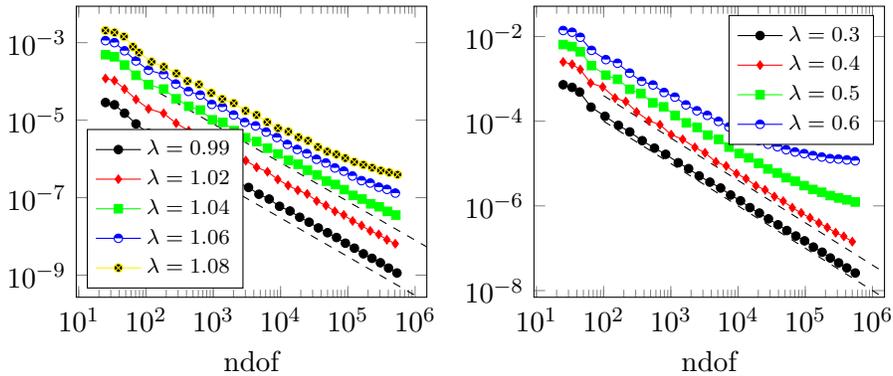
\subsubsection*{Experiment 2 (Continuous variable exponent)}
In this subsection we investigate the second example in Section \ref{sec:ExamplesLavrentiev}. We set the exponents $p_- = 3/2$ and $p_+=3$. The initial triangulation resolves the domains for $p_-$ and $p_+$ and we use a minimum quadrature rule, that is, we choose the point $x_T\in T\in \mathcal{T}$ in \eqref{eq:Assumptionintegrand} as the minimiser $x_T = \argmin_T p$. Hence, \eqref{eq:Assumptionintegrand}--\eqref{eq:nabla2} is satisfied. Figure \ref{fig:Exp2} displays the approximated energies $\mathcal{F}_h(u_\textup{nc})$ of the $W$-solution and $\mathcal{F}(u_\textup{c})$  of the $H$-solution for various scaling parameters $\lambda$. These numerical results suggest a Lavrentiev gap for $\lambda \geq 3$.

In contrast to the previous experiment, the approximated energy $\mathcal{F}_h(u_\textup{nc})$ converges with the rate $\mathcal{O}(\textup{ndof}^{-1/2})$. This reduced rate is caused by the minimum quadrature rule; the loss of midpoint symmetry reduces the order of accuracy. Numerical experiments indicate that the energy $\mathcal{F}(u_\textup{nc})$ does, in contrast to the energy $\mathcal{F}_h(u_\textup{nc})$, \textit{not} converge to the $W$-minimiser. However, computing the minimum $\min_{\CR} \mathcal{F}$ over the non-conforming space with a more accurate quadrature rule leads to faster convergence of the approximated energy to the $W$-minimiser. 
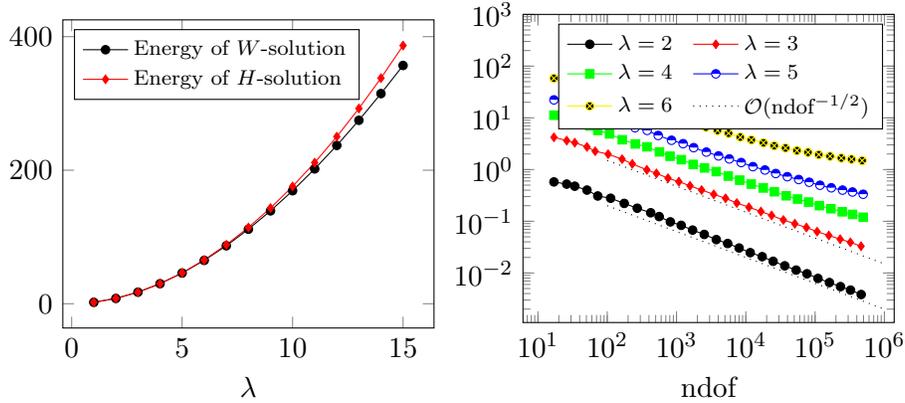
\begin{figure}
\begin{center}
\begin{tikzpicture}
\begin{axis}[
xlabel={$\lambda$},
cycle multi list={\nextlist MyColors},
scale = {.72},
legend cell align=left,
legend style={font=\scriptsize, legend columns=1,legend pos=north west}
]
	\addplot table [x=t,y=EnergyCR]{ExperimentsPaper/Data/Exp2/bulk_0.3.txt};
	\addplot table [x=t,y=EnergyCexact]{ExperimentsPaper/Data/Exp2/bulk_0.3.txt};
		\legend{
	{Energy of $W$-solution},{Energy of $H$-solution}};
\end{axis}
\end{tikzpicture}
\begin{tikzpicture}
\begin{axis}[
xmode = log,
ymode = log,
ymax = 1000,
legend columns=2, 
xlabel={ndof},
cycle multi list={\nextlist MyColors},
scale = {.72},
legend cell align=left,
legend style={font=\scriptsize, legend columns=1,legend pos=north east}
]
	\addplot table [x=ndof,y=DistVal] {ExperimentsPaper/Data/Exp2/t_2_bulk_0.3Calc.txt};
	\addplot table [x=ndof,y=DistVal] {ExperimentsPaper/Data/Exp2/t_3_bulk_0.3Calc.txt};
		\addplot table [x=ndof,y=DistVal] {ExperimentsPaper/Data/Exp2/t_4_bulk_0.3Calc.txt};
		\addplot table [x=ndof,y=DistVal] {ExperimentsPaper/Data/Exp2/t_5_bulk_0.3Calc.txt};
	\addplot table [x=ndof,y=DistVal] {ExperimentsPaper/Data/Exp2/t_7_bulk_0.3Calc.txt};
		\addplot[dotted,sharp plot,update limits=false] coordinates {(1e2,1.5e0) (1e6,1.5e-2)};
	\addplot[dotted,sharp plot,update limits=false] coordinates {(1e2,2e-1) (1e6,2e-3)};
	\legend{
	{$\lambda = 2\ \ $},{$\lambda = 3$},{$\lambda = 4\ \ $},{$\lambda = 5$},{$\lambda = 6\ \ $},{$\mathcal{O}(\textup{ndof}^{-1/2})$}};
\end{axis}
\end{tikzpicture}
\caption{Computed energies of $W$- and $H$-solution (left) and the distance $\mathcal{F}(u_\textup{c}) - \mathcal{F}_h(u_\textup{nc})$ for various scalings $\lambda$}\label{fig:Exp2}
\end{center}
\end{figure}
\subsubsection*{Experiment 3 (Double phase potential)}
In this experiment we approximate the minimisers of the energy in the third example in Section \ref{sec:ExamplesLavrentiev} with parameters $\alpha = 0$ as well as $p_- = 3/2$ and $p_+=3$. The initial triangulation resolves the weight function $a
$, allowing for an exact quadrature with piece-wise constant weights. Thus, \eqref{eq:Assumptionintegrand}--\eqref{eq:nabla2} hold true. The right-hand side in Figure \ref{fig:exp1} displays the convergence history plot of the energies $\mathcal{F}(u_\textup{c}) - \mathcal{F}(u_\textup{nc})$. The plot indicates a Lavrentiev gap for $\lambda \geq 0.4$.
\subsubsection*{Experiment 4 (Borderline case of double phase potential)}
In this experiment we investigate the fourth example of Section \ref{sec:ExamplesLavrentiev} with parameters $\beta = \gamma = 2$. As in the previous experiments we use an initial triangulation that resolves the weight function $a$, allowing for exact quadrature and \eqref{eq:Assumptionintegrand}--\eqref{eq:nabla2}.
Since the Newton scheme struggles with the computation of the discrete minimiser, we utilise a fixed point iteration similar to the one introduced in \cite{DieningFornasierWank17} without regularisation. The paper \cite{BalSur21} proves that the $W$-minimum grows asymptotically slightly slower than the $H$-minimum with respect to the scaling of the boundary data $\psi = \lambda u_0$, leading to a gap for all sufficiently large parameters $\lambda$. Our numerical experiments, displayed in Figure \ref{fig:Exp4a}, indicate a gap for all $\lambda >0$. Moreover, the $W$- and $H$-solution differ significantly.
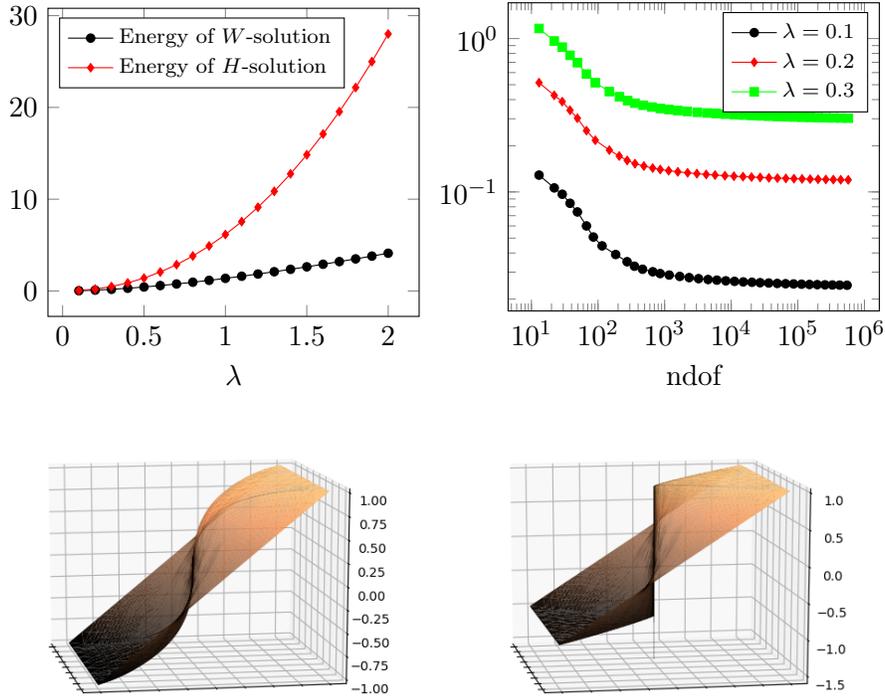
\begin{figure}
\begin{center}
\begin{tikzpicture}
\begin{axis}[
xlabel={$\lambda$},
cycle multi list={\nextlist MyColors},
scale = {.72},
legend cell align=left,
legend style={font=\scriptsize, legend columns=1,legend pos=north west}
]
	\addplot table [x=t,y=EnergyCR]{ExperimentsPaper/Data/Exp4/bulk_0.3.txt};
	\addplot table [x=t,y=EnergyC]{ExperimentsPaper/Data/Exp4/bulk_0.3.txt};
		\legend{
	{Energy of $W$-solution},{Energy of $H$-solution}};
\end{axis}
\end{tikzpicture}
\begin{tikzpicture}
\begin{axis}[
xmode = log,
ymode = log,
xlabel={ndof},
cycle multi list={\nextlist MyColors},
scale = {.72},
legend cell align=left,
legend style={font=\scriptsize, legend columns=1,legend pos=north east}
]
	\addplot table [x=ndof,y=DistVal] {ExperimentsPaper/Data/Exp4/t_0.1_bulk_0.3Calc.txt};
	\addplot table [x=ndof,y=DistVal] {ExperimentsPaper/Data/Exp4/t_0.2_bulk_0.3Calc.txt};
		\addplot table [x=ndof,y=DistVal] {ExperimentsPaper/Data/Exp4/t_0.3_bulk_0.3Calc.txt};
		\legend{
	{$\lambda = 0.1$},{$\lambda = 0.2$},{$\lambda = 0.3$}};
\end{axis}
\end{tikzpicture}
\hspace*{-.8cm}
\includegraphics[scale=.48]{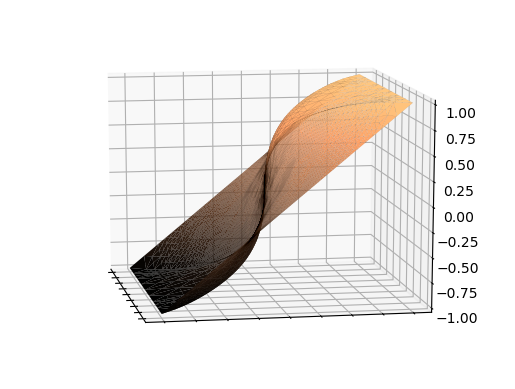}
\hspace*{-.4cm}
\includegraphics[scale=.48]{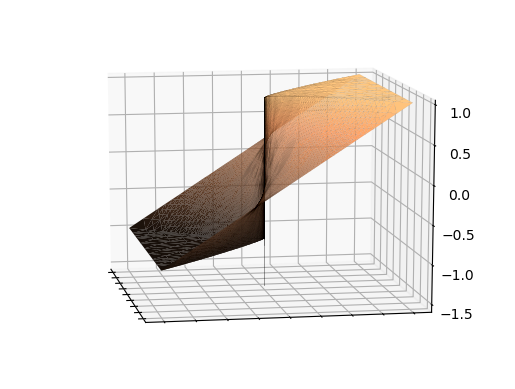}
\caption{Energies of $W$-solution and $H$-solution (top left), the distance $\mathcal{F}(u_\textup{c}) - \mathcal{F}(u_\textup{nc})$ for various scalings $\lambda$ (top right), and the $H$-minimiser $u_\textup{c}$ (bottom left) as well as the $W$-minimiser $u_\textup{nc}$ (bottom right) for $\lambda = 1$}\label{fig:Exp4a}
\end{center}
\end{figure}

\subsection{Bad Quadrature/Geometric Regularisation}\label{subsec:BadQuad}
This experiment investigates the importance of an appropriate quadrature rule. Therefore, we perturb the quadrature in Experiment 1. Our initial triangulation resolves the piece-wise constant exponent $p(\cdot)$. Thus, it has a node in the origin $0\in \Omega$. Let $\omega(0) \subset \mathcal{T}$ denote the nodal patch with respect to this node, that is, the set of all triangles $T\in \mathcal{T}$ with $0\in T$. Our perturbed piece-wise constant exponents read
\begin{align*}
p_\textup{max}|_T = \begin{cases}
p_+ &\text{for } T\in \omega(0),\\
 p&\text{for } T\in \mathcal{T} \setminus \omega(0)
\end{cases}\quad \text{and}\quad
p_\textup{min}|_T = \begin{cases}
p_- &\text{for } T\in \omega(0),\\
 p&\text{for } T\in \mathcal{T} \setminus \omega(0).
\end{cases}
\end{align*}
Notice that these approximations converge to the exact exponent $p$ as the mesh is refined.
We minimise the functionals
\begin{align*}
\mathcal{F}_\textup{max} = \int_\Omega \frac{1}{p_\textup{max}(x)}|\nabla_h \cdot|^{p_\textup{max}(x)}\,\mathrm{d}x,\qquad\mathcal{F}_\textup{min} = \int_\Omega \frac{1}{p_\textup{min}(x)}|\nabla_h \cdot|^{p_\textup{min}(x)}\,\mathrm{d}x
\end{align*}
over the Lagrange and Crouzeix-Raviart space with boundary data $\psi = \lambda u_0$ and $\lambda = 5$. The corresponding solutions read $u_\textup{c}^\textup{max} \in \mathcal{L}^1_1(\mathcal{T})$ and $u_\textup{nc}^\textup{max} \in \CR$ as well as $u_\textup{c}^\textup{min} \in \mathcal{L}^1_1(\mathcal{T})$ and $u_\textup{nc}^\textup{min} \in \CR$. Figure \ref{fig:ExpBadQuad} compares the resulting energies with a reference solution computed in Experiment 1 on the finest triangulation. It indicates that the minimisation of the energy $\mathcal{F}_\textup{max}$ over the conforming and non-conforming space leads to the $H$-minimiser; the minimisation of the energy $\mathcal{F}_\textup{min}$ over the conforming and non-conforming space leads to the $W$-minimiser. 
This shows that the conforming scheme requires exact or at least some suitable quadrature to converge to the $H$-minimiser. Moreover, our computations indicate that it might be possible to design a conforming scheme that converges to the $W$-minimiser by introducing a regularisation near singularities. However, the adaptive scheme experiences difficulties after $\textup{ndof} = \dim \CR$ exceeds $10^4$. Thus, its convergence to the exact $W$-minimiser is unclear.
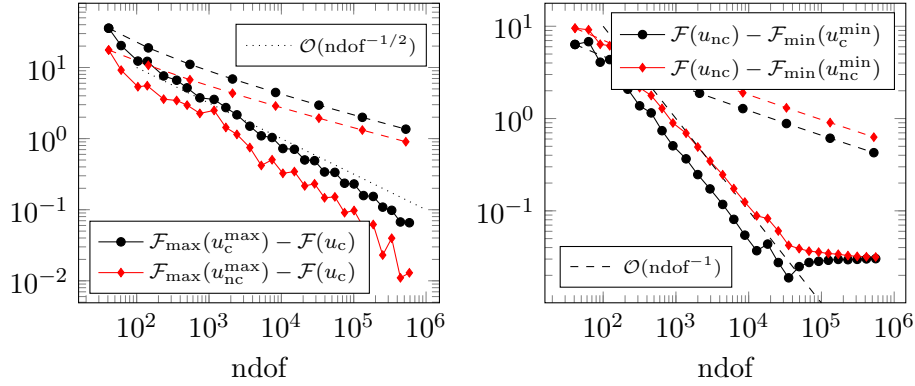
\begin{figure}
\begin{center}
\begin{tikzpicture}
\begin{axis}[
xlabel={$\textup{ndof}$},
xmode = log,
ymode = log,
cycle multi list={\nextlist MyColors},
scale = {.7},
legend cell align=left,
legend style={font=\scriptsize, legend columns=1,legend pos=south west}
]
	\addplot table [x=ndof,y=distC]{ExperimentsPaper/Data/ExpBadQuad/BadQuad_MaxTrue_t_5_bulk_0.3Calc.txt};
	\addplot table [x=ndof,y=distCR]{ExperimentsPaper/Data/ExpBadQuad/BadQuad_MaxTrue_t_5_bulk_0.3Calc.txt};
	\addplot[dotted,sharp plot,update limits=false] coordinates {(1e2,1e1) (1e6,1e-1)};\label{line:dotted}
				\pgfplotsset{cycle list shift=2};
	\addplot table [x=ndof,y=distC]{ExperimentsPaper/Data/ExpBadQuad/BadQuad_MaxTrue_t_5_bulk_1Calc.txt};
	\addplot table [x=ndof,y=distCR]{ExperimentsPaper/Data/ExpBadQuad/BadQuad_MaxTrue_t_5_bulk_1Calc.txt};	
		\legend{
	{$\mathcal{F}_\textup{max}(u_\textup{c}^\textup{max})-\mathcal{F}(u_\textup{c}) $},{$\mathcal{F}_\textup{max}(u_\textup{nc}^\textup{max})-\mathcal{F}(u_\textup{c})$}};
	\node [draw,fill=white,font=\fontsize{7}{5}\selectfont] at (rel axis cs: 0.7,0.86) {\shortstack[l]{
\ref{line:dotted} $\mathcal{O}(\textup{ndof}^{-1/2})$}};
\end{axis}
\end{tikzpicture}
\begin{tikzpicture}
\begin{axis}[
xlabel={$\textup{ndof}$},
xmode = log,
ymode = log,
cycle multi list={\nextlist MyColors},
scale = {.7},
legend cell align=left,
legend style={font=\scriptsize, legend columns=1,legend pos=north east}
]
	\addplot table [x=ndof,y=distC]{ExperimentsPaper/Data/ExpBadQuad/BadQuad_MaxFalse_t_5_bulk_0.3Calc.txt};
	\addplot table [x=ndof,y=distCR]{ExperimentsPaper/Data/ExpBadQuad/BadQuad_MaxFalse_t_5_bulk_0.3Calc.txt};	
		\addplot[dashed,sharp plot,update limits=false] coordinates {(1e2,1e1) (1e5,1e-2)};
				\pgfplotsset{cycle list shift=2};
	\addplot table [x=ndof,y=distC]{ExperimentsPaper/Data/ExpBadQuad/BadQuad_MaxFalse_t_5_bulk_1Calc.txt};
	\addplot table [x=ndof,y=distCR]{ExperimentsPaper/Data/ExpBadQuad/BadQuad_MaxFalse_t_5_bulk_1Calc.txt};	

		\legend{
	{$\mathcal{F}(u_\textup{nc}) - \mathcal{F}_\textup{min}(u_\textup{c}^\textup{min})$},{$\mathcal{F}(u_\textup{nc}) - \mathcal{F}_\textup{min}(u_\textup{nc}^\textup{min})$}};
		\node [draw,fill=white,font=\fontsize{7}{5}\selectfont] at (rel axis cs: 0.28,0.12) {\shortstack[l]{
\ref{line:dashed} $\mathcal{O}(\textup{ndof}^{-1})$}};
\end{axis}
\end{tikzpicture}
\caption{Convergence of computed energies with adaptive (solid line) and uniformly (dashed line) refined meshes to reference value from Experiment 1}\label{fig:ExpBadQuad}
\end{center}
\end{figure}

\subsection{Multiple Saddle Points}\label{eq:ExpLavrGap}
The last example explores the Lavrentiev gap phenomenon for a problem with three saddle points. More precisely, we compute the $W$- and $H$-minimiser of the variable exponent $p(\cdot)$-Laplacian $\int_\Omega |\nabla_h\cdot |^{p(x)}/p(x)\dx$ on the domain $\Omega = (-1,5)\times (-1,1)$ with boundary data $\psi(x_1,x_2) = \lambda x_2$ for all $(x_1,x_2) \in \overline{\Omega}$ and parameters $\lambda >0$. The piece-wise constant exponent (visualised in Figure~\ref{fig:MultiSat}) attains the values $p_- = 3/2$ and $p_+ = 3$ and reads for all $(x_1,x_2)\in \Omega$
\begin{align*}
p(x_1,x_2) = \begin{cases}
 p_-&\text{for } |x_1| < |x_2|\text{ and }x_1<1,\\
 p_-&\text{for } |x_1-2| < |x_2|\text{ and }1\leq x_1<3,\\
 p_-&\text{for } |x_1-4| < |x_2|\text{ and }3\leq x_1,\\
 p_+&\text{else}.
\end{cases}
\end{align*} 
For $\lambda = 1$ the exact minimiser reads $u(x_1,x_2) = x_2$ for all $(x_1,x_2) \in \Omega$.
The initial triangulation resolves the exponent $p$. Thus, we can apply an exact quadrature with piece-wise constant exponents satisfying \eqref{eq:Assumptionintegrand}--\eqref{eq:nabla2}. 
Figure \ref{fig:MultiSat} displays a convergence history plot of the energies for various $\lambda$ as well as a plot of the $H$- and $W$-minimiser for $\lambda = 5$.
As in Experiment 1, it seems that there is a gap for $\lambda >1$. The $W$-minimiser seems to jump for $\lambda >1$ in all three saddle points.
\begin{figure}
  \centering
\begin{tikzpicture}[scale = .85]
    \node at (2,1.2) {Variable exponent~$p$};
    \draw[dashed] (-1,-1) -- (-1,+1) -- (+5,+1) -- (+5,-1) --cycle;
    \draw (-1,-1) -- (1,1);
    \draw (1,-1) -- (-1,1);
    \draw (1,1) -- (3,-1);
    \draw (1,-1) -- (3,1);
    \draw (3,1) -- (5,-1);
    \draw (3,-1) -- (5,1);

    \node at (4,0.6) {$p_+$};
    \node at (4,-0.6) {$p_+$};
    \node at (2,0.6) {$p_+$};
    \node at (2,-0.6) {$p_+$};
    \node at (0,0.6) {$p_+$};
    \node at (0,-0.6) {$p_+$};
    \node at (-0.6,0) {$p_-$};
    \node at (1,0) {$p_-$};
    \node at (3,0) {$p_-$};
    \node at (4.6,0) {$p_-$};

    \node[left] at (-1,1) {$1$};
    \node[left] at (-1,-1) {$-1$};
    \node[left] at (-1,0) {$0$};
    
    \node[below] at (-1,-1) {$-1$};
    \node[below] at (1,-1) {$1$};
    \node[below] at (3,-1) {$3$};
    \node[below] at (5,-1) {$5$};

    \node at (0,-3) {};
\end{tikzpicture}
\begin{tikzpicture}
\begin{axis}[
xmode = log,
ymode = log,
xlabel={ndof},
cycle multi list={\nextlist MyColors},
scale = {0.72},
legend cell align=left,
legend style={font=\scriptsize, legend columns=1,legend pos=north east}
]
	\addplot table [x=ndof,y=DistVal] {ExperimentsPaper/Data/ExpMultSaddles/t_1.05_bulk_0.3Calc.txt};
	\addplot table [x=ndof,y=DistVal] {ExperimentsPaper/Data/ExpMultSaddles/t_1.1_bulk_0.3Calc.txt};
	\addplot table [x=ndof,y=DistVal] {ExperimentsPaper/Data/ExpMultSaddles/t_1.15_bulk_0.3Calc.txt};
	\addplot[dashed,sharp plot,update limits=false] coordinates {(1e2,1e-3) (1e7,1e-8)};
	\legend{
	{{$\lambda = 1.05$},{$\lambda = 1.1$},{$\lambda = 1.15$},{$\mathcal{O}(\textup{ndof}^{-1})$}}};
\end{axis}
\end{tikzpicture}
\hspace*{-.8cm}
\includegraphics[scale=.5]{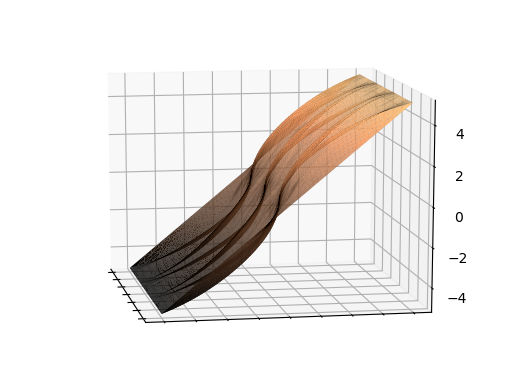}
\includegraphics[scale=.5]{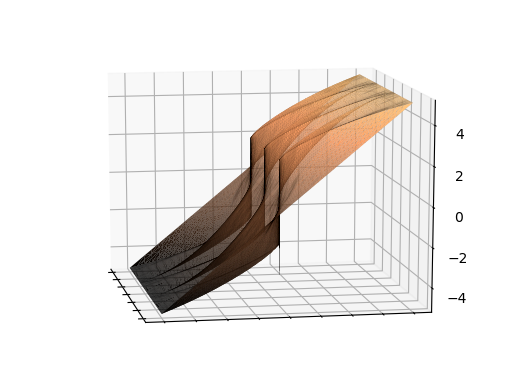}
\caption{Variable exponent $p$ (top left),  the distance $\mathcal{F}(u_\textup{c}) - \mathcal{F}(u_\textup{nc})$ for various scalings $\lambda$ (top right), and the $H$-minimiser $u_\textup{c}$ (bottom left) as well as the $W$-minimiser $u_\textup{nc}$ (bottom right) for $\lambda = 5$}\label{fig:MultiSat}
\end{figure}
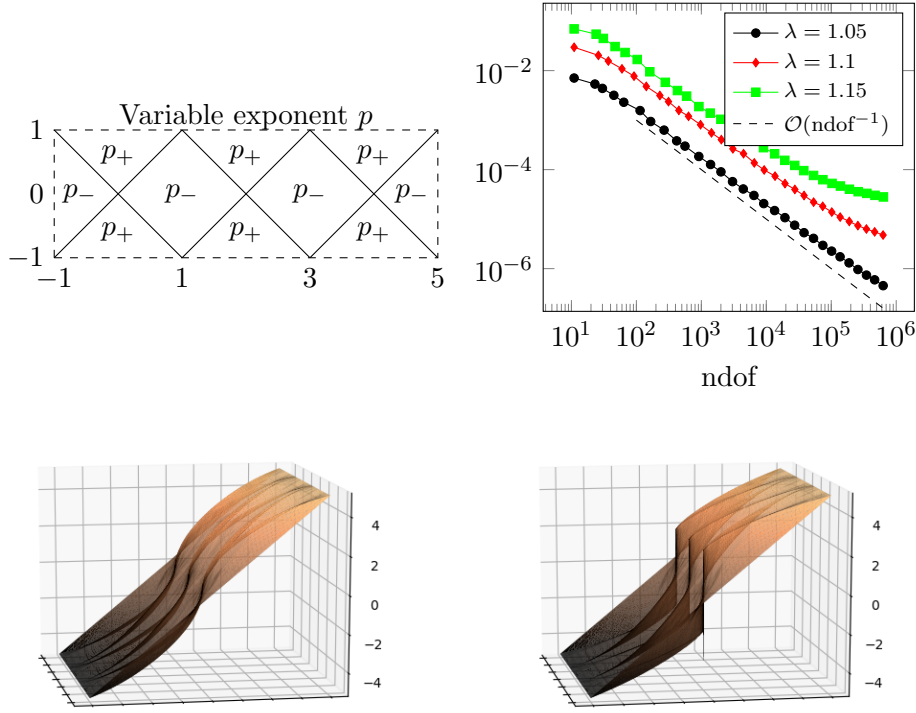

\section{Conclusion}
In this paper we have successfully adapted the Crouziex--Raviart finite element scheme to approximate variational problems with non-autonomous integrands even in the presence of the Lavrentiev gap phenomenon. We have identified assumptions on which convergence is guaranteed, and have demonstrated numerically on wide range of test cases that the scheme is practical and can reliably predict the existence (or non-existence) of Lavrentiev gaps. The examples we considered here are of primary interest to the theoretical study of regularity of solutions to variational problems. Indeed we hope that our numerical scheme could be employed more generally towards refining and extending the analytical results based on which we chose our examples. 

More generally, however, our results provide strong new evidence for the advantages of non-conforming methods in the numerical solution of difficult variational problems. On that theme, we note that a long-standing open problem is the extension of our convergence results to poly-convex or even quasi-convex integrands.
Further investigation might involve the use of problem dependent strategies, as for example done in~\cite{BreDieSch15,DieningRuszika07,EbmeyerLiu05} for conforming and in~\cite{Bar21,Gudi10,CarstensenLiu15,CarstensenPeterseimSchedensack12} for non-conforming schemes for problems without Lavrentiev gap, to conclude rates of convergence and to relax some assumptions like the growth condition \eqref{eq:growth_conditions} for specific problems. 
\appendix
\section{Variational Calculus Background}
This paper, in particular Section \ref{sec:ProofOfLemma3}, relies on numerous analytical results. We present these results and needed calculations in this appendix.
Most of these calculations rely the following properties of convex functions $\phi,\chi$ and their convex conjugates $\phi^*,\chi^*$ (see \eqref{eq:defConvConj} for a definition).
\begin{proposition} \label{lem:conj}
Let~$a,b>0$ be positice numbers, then there holds
\begin{enumerate}
\item $(a\, \phi(b \xi))^* = a\, \phi^*(\xi/(ab))$ for all $\xi\in \mathbb{R}^{m\times n}$,
\item if~$\chi\le \phi$, then~$\phi^* \leq \chi^*$,
\item $(\phi+a)^*=\phi^*-a$.
\end{enumerate}
\end{proposition}
\begin{proof}
These properties follow from the definition of the convex conjugate and simple calculations, see~\cite{KraRut61} or~\cite[Lem.\ 2.4.3]{HarHas19} for details.
\end{proof}

In the remainder of this appendix $\phi:\Omega\times \mathbb{R}^{m\times n}\to \mathbb{R}$ is convex in its second component and satisfies the following generalised version of the growth condition in \eqref{eq:growth_conditions}. There exists a function $c_0 \in L^1(\Omega)$ and positive constants $c_1,c_2<\infty$ such that for almost all $x\in \Omega$ and all $\xi \in \mathbb{R}^{m\times n}$
\begin{align}
\label{eq:growth_conditions2}
-c_0(x) +c_1 \abs{\xi}^{p_-}\le \phi(x,\xi)\le c_2\abs{\xi}^{p_+} +c_0(x).
\end{align}

In addition, certain results require the following generalised versions of the assumptions in \eqref{eq:Delta_2phi} and \eqref{eq:nabla2}.
\begin{definition}[$\Delta_2$- and $\nabla_2$-condition]\label{def:gener_del}
We say that~$\phi$ satisfies the generalised~$\Delta_2$-condition if there exists a constant~$C$ and a function~$\tilde{C}\in L^1(\Omega)$ for all~$\xi \in \mathbb{R}^{m\times n}$ and almost all $x\in \Omega$  such that
\begin{align}\label{eq:Delta2Gen}
    \phi(x,2\xi)\le C\phi(x,\pm\xi)+\tilde{C}(x).
\end{align}
We say that~$\phi$ satisfies the generalised~$\nabla_2$-conditions if~$\phi^*$ satisfies the generalised~$\Delta_2$-condition, that is, there exists a constant~$K$ and a function~$\tilde{K}\in L^1(\Omega)$ for all~$\xi \in \mathbb{R}^{m\times n}$ and almost all $x\in \Omega$  such that
\begin{align}\label{eq:Nabla2Gen}
    \phi^*(x,2\xi)\le K\phi^*(x,\pm\xi)+\tilde{K}(x).
\end{align}
\end{definition}
The following proposition shows the equivalence of the assumption in \eqref{eq:nabla2} and the~$\nabla_2$-condition.
\begin{proposition}[$\nabla_2$-condition]\label{lem:Nabla}
Let $x\in \Omega$ and $\xi \in \mathbb{R}^{m\times n}$. Then the inequality in \eqref{eq:Nabla2Gen} is equivalent to
\begin{align}\label{eq:A2Gen}
    K\phi(x,2 \xi)& \leq \phi(x, \pm K \xi)+\tilde{K}(x).
\end{align}

\end{proposition}
\begin{proof}
Let $x \in \Omega$ and $\xi \in \mathbb{R}^{m\times n}$.
Lemma \ref{lem:conj} and \eqref{eq:A2Gen} yield
\begin{align*}
&\phi^*\left(x,\pm K^{-1} \xi\right) - \tilde{K}(x) = (\phi(x,\pm K \xi) + \tilde{K}(x))^*\\
&\quad \leq (K\, \phi(x,2\xi))^* = K\,\phi^*\left(x, (2K)^{-1} \xi \right).
\end{align*}
Substituting $\xi$ by $2K\xi$ implies \eqref{eq:Nabla2Gen}. A similar calculation shows that \eqref{eq:Nabla2Gen} implies \eqref{eq:A2Gen}.
\end{proof}
In order to include boundary conditions and right-hand side, we have to shift the integrand. The following propositions shows that these shifts do not cause significant difficulties, since the resulting integrands are equivalent up to a $L^1(\Omega)$-function $\tilde{c}_1$.
\begin{proposition}[Weak equivalence]\label{prop:A2}
Suppose that the integrand $\phi$ satsifies the $\nabla_2$-condition \eqref{eq:Nabla2Gen}.
Let $F \in L^{p_{-}'}(\Omega;\setR^{m\times n})$ and define the integrand 
\begin{align*}
    \phi_1(x,\xi):=\phi(x,\xi)+F(x):\xi\qquad\text{for all }x\in \Omega \text{ and }\xi \in \mathbb{R}^{m\times n}.
\end{align*}
Then there exist a function $\tilde{c}_1 \in L^1(\Omega)$ such that we have for almost all $x\in \Omega$ and all $\xi \in \mathbb{R}^{m\times n}$
\begin{align*}
-\tilde{c}_1(x)+\frac 12\phi(x,\xi)\le \phi_1(x,\xi) \le \frac 32 \phi(x,\xi) +\tilde{c}_1(x).
\end{align*}
\end{proposition}
\begin{proof}
Let $x\in \Omega$ and $\xi \in \mathbb{R}^{m\times n}$.
  Young's inequality, the $\nabla_2$-condition \eqref{eq:Nabla2Gen}, and the growth \eqref{eq:growth_conditions2} (leading to growth conditions for $\phi^*$, cf.~\eqref{eq:GrowthDual}) yield
  \begin{align*}
      &2\abs{F(x):\xi} \le  \abs{F(x):2\xi} \le \phi(x,\xi)+\phi^*(x,2F(x))\\
      &\le (\phi(x,\xi)+K\phi^*(x,F(x)) + \tilde{K}(x))\le \phi(x,\xi) + K c_1^{1-p_-'}\abs{F(x)}^{p_{-}'} + K {c}_0(x).
  \end{align*}
Thus, the function $\tilde{c}_1 \coloneqq K( c_1^{1-p_-'}\abs{F}^{p_{-}'} +{c}_0)/2 \in L^1(\Omega)$ satisfies
  \begin{align*}
 -\frac 12 \phi(x,\xi) -\tilde{c}_1(x)
     \le F(x):\xi \le  \frac 12 \phi(x,\xi) +\tilde{c}_1(x).
  \end{align*}
  Using this inequalities and the definition of $\phi_1$ concludes the proof.
\end{proof}
The following result shows that pointwise convergence implies pointwise convergence of the convex conjugate.
\begin{proposition}[{\cite[Lem.\ 5.4]{PasKhBal11}}]\label{lem:conjuageconv}
For all $h>0$ let $\phi_h:\Omega \times \mathbb{R}^{m\times n} \to \mathbb{R}$ be a function that satisfies the nonstandard growth condition~\eqref{eq:growth_conditions2} and that is is convex in its second component. Then pointwise convergence of the primal functions~$\phi_h$ implies pointwise convergence of the corresponding convex conjugates~$\phi^*_h$ in the sense that for $x\in \Omega$ and $\xi \in \mathbb{R}^{m\times n}$
\begin{align*}
    \lim_{h \to 0} \phi_h(x,\xi) = \phi (x,\xi)\qquad  \text{implies}  \qquad \lim_{h\to 0} \phi^*_h(x,\xi) = \phi^*(x,\xi). 
\end{align*}
\end{proposition}
We conclude this appendix with a result involving  the relaxation (see \eqref{eq:RelaxedG}) of the energy
$ \mathcal{G} \coloneqq \int_\Omega \phi^*(x,\cdot)\dx$. 

Notice that Zhikov states the result under  $\Delta_2$-condition (the $L^1(\Omega)$ function in \eqref{eq:Delta2Gen} is replaced by a constant) in \cite{Zhi11}. A review of his proof shows that the result is valid under the weaker assumption in \eqref{eq:Delta2Gen} as well, leading to the following result. 
\begin{proposition}[{\cite[Thm.\ 11.10]{Zhi11}}]\label{thm:Zhi1110}
If~$\phi$ satisfies the~$\Delta_2$-condition \eqref{eq:Delta2Gen}, then it holds that
\begin{align*}
    \overline{\mathcal{G}}= \mathcal{G}\qquad \text{on }\dom \overline{\mathcal{G}}.
\end{align*}
\end{proposition}

\printbibliography

\end{document}